\documentclass[11pt,english]{article}
\usepackage[T1]{fontenc}
\usepackage[latin9]{inputenc}
\usepackage{geometry}
\usepackage{mathrsfs}
\usepackage{amsmath}
\usepackage{amsthm}
\usepackage{amssymb}

\makeatletter

\theoremstyle{theorem}
    \newtheorem{thm}{Theorem}
\theoremstyle{definition}
    \newtheorem{defn}{Definition}[section]
\theoremstyle{remark}
    \newtheorem{rem}[defn]{Remark}
\theoremstyle{definition}
  \newtheorem{notation}[defn]{Notation}
\theoremstyle{plain}
  \newtheorem{prop}[defn]{Proposition}
\theoremstyle{plain}
  \newtheorem{lem}[defn]{Lemma}
\theoremstyle{definition}
  \newtheorem{example}[defn]{Example}
\theoremstyle{remark}
  \newtheorem{claim}[defn]{Claim}
\theoremstyle{plain}
  \newtheorem{cor}[defn]{Corollary}
\theoremstyle{theorem}
  \newtheorem{open}[defn]{Open problem}

\usepackage[hidelinks]{hyperref}
\date{}

\usepackage[textsize=footnotesize,textwidth=0.85in]{todonotes}

\usepackage[mathscr]{euscript}
\DeclareFontFamily{OT1}{pzc}{}
\DeclareFontShape{OT1}{pzc}{m}{it}{<-> s * [0.900] pzcmi7t}{}
\DeclareMathAlphabet{\mathpzc}{OT1}{pzc}{m}{it}

\usepackage{tikz}
\usepackage{caption}
\usepackage{subcaption}

\usepackage{authblk}

\makeatother

\usepackage{babel}

\begin{document}
\global\long\def\ampli{\mathcal{A}_{n,k,m}}
\global\long\def\amplio{\mathcal{A}_{n,k,m}^{>0}}
\global\long\def\sub#1{\binom{\left[n\right]}{#1}}
\global\long\def\conv#1{{\rm Conv}\left(#1\right)}
\global\long\def\convo#1{{\rm Conv}^{\circ}\left(#1\right)}
\global\long\def\spanp#1{{\rm Span}_{\geq0}\left(#1\right)}
\global\long\def\gr#1{{\rm Gr}_{#1}}
\global\long\def\grp#1{{\rm Gr}_{#1}^{\geq0}}
\global\long\def\grsp#1{{\rm Gr}_{#1}^{>0}}
\global\long\def\grwcb#1{{\rm Gr}_{#1}^{{\rm wcb}}}
\global\long\def\matp{{\rm Mat}_{n,k+m}^{>0}}
\global\long\def\ampliwb{\mathcal{A}_{n,k,m}^{{\rm wcb}}}

\title{The amplituhedron crossing and winding numbers}

\author{Xavier Blot\footnote{Jacob Ziskind Building, Weizmann Institute of Science Rehovot 76100, Israel. email: xavier.blot@weizmann.ac.il} , Jian-Rong Li\footnote{Faculty of Mathematics, University of Vienna, Oskar-Morgenstern-Platz 1, 1090 Vienna, Austria.
email: lijr07@gmail.com  }}

\maketitle
\begin{abstract}
	In \cite{arkani2018unwinding}, Arkani-Hamed, Thomas and Trnka formulated two conjectural descriptions of the tree amplituhedron $\ampli$ depending on the parity of $m$. When $m$ is even, the description involves the winding number and when $m$ is odd the description involves the crossing number. In this paper, we prove that if a point of the amplituhedron is in the image of the positive Grassmannian by the amplituhedron map, then it satisfies the winding or crossing descriptions depending on the parity of $m$. When $m=2$, we also prove the other direction: a point satisfying the winding description is inside the amplituhedron.
\end{abstract}
\tableofcontents{}

\section{Introduction}

\subsection{Overview}

The totally nonnegative (TNN) Grassmannian ${\rm Gr}_{k,n}^{\geq0}$ is the subspace of the Grassmannian ${\rm Gr}_{k,n}$ given by the $k$-dimensional vector spaces of $\mathbb{R}^n$ with nonnegative Pl\"ucker coordinates. It has relations with many areas of mathematics including integrable systems \cite{kodama2014kp} or tropical geometry \cite{speyer2005tropical}. 



 The (tree) amplituhedron is generalization of the TNN Grassmannian introduced in 2013 by Arkani-Hamed and Trnka \cite{arkani2014amplituhedron} to study scattering amplitudes in physics from a geometrical point of view. It has attracted a lot of interest both in high-energy physics and in mathematics. As a consequence, ideas from physics gave rise to beautiful mathematical developments such as the study of the BCFW triangulation \cite{britto2005new, britto2005direct, tessler2021amplituhedron} or the exploration of the T-duality \cite{lukowski2020positive, parisi2021m}. It is an intriguing object finding connections with other areas of mathematics such as cluster algebras \cite{lukowski2019cluster, parisi2021m} or tropical geometry \cite{lukowski2020positive}. See also \cite{williams2021positive} for a recent survey. 
 
 

 


The amplituhedron in its original definition \cite{arkani2014amplituhedron} is given by the image of the totally nonnegative Grassmannian $\grp{k,n}$, inside an ambient Grassmannian $\gr{k,k+m}$, where $m \leq n-k$, by a (map induced by) linear  map.
Although being simple, this definition is not fully satisfying. First, this definition is very redundant since  the dimension of the source Grassmannian is bigger than the dimension of the ambient Grassmannian. Second, it is very hard to check whether or not a point of ambient Grassmannian $\gr{k,k+m}$ belongs to the amplituhedron. 
To overcome the first point, one could select cells in the source TNN Grassmannian such that they map injectively to the amplituhedron, and such that their images triangulate the amplituhedron. There have been a lot of effort in that direction, see for instance \cite{karp2019amplituhedron} when $m=1$,  or \cite{bao2019m, karp2020decompositions, parisi2021m} for $m=2$ and \cite{tessler2021amplituhedron} when $m=4$, but this is not the path we follow here.


In \cite{arkani2018unwinding} Arkani-Hamed, Thomas and Trnka conjectured three new  definitions of the amplituhedron purely in terms of combinatorial and topological data. To understand the new point of view of these definitions, one has to regard the amplituhedron as a generalization of a convex polytope inside the ambient Grassmannian. It is well-know that convex polytopes have two descriptions; either as convex hull of its vertices, or as a finite intersection of half-spaces. The original definition of the amplituhedron uses the first point of view, and the new definitions of \cite{arkani2018unwinding} use the second. More precisely, the points of the amplituhedron satisfy a set of inequalities and one wants to interpret these inequalities as codimension-one faces of the amplituhedron inside the ambient Grassmannian. However, in general, these inequalities do not determine the amplituhedron. To fully determine the amplituhedron, the authors introduce three numbers associated to each point of the ambient Grassmannian: the crossing number (for odd $m$), the winding number (for even $m$) and the number of sign flips (for all $m$). All these numbers are interrelated, and defined according to combinatorial and topological properties of the point of the ambient Grassmannian. Then, Arkani-Hamed, Thomas and Trnka conjecture that a point of the ambient Grassmannian is inside the amplituhedron if and only if it satisfies the boundary inequalities and one of the three topological and combinatorial numbers has a definite value.





This completely new perspective on the amplituhedron offers numerous compelling advantages, and new avenues to tackle the amplituhedron. We give here a few motivations. 
First, it is an intrinsic definition, that is the amplituhedron is directly defined as a subset of the ambient Grassmannian and does not require external information such as the TNN grassmannian, in particular it not redundant as the original definition. 
Second, the three numbers: crossing, winding, and sign flips can be computed simply by analyzing the sign configurations of the so-called \emph{twistor coordinates}. Moreover, the boundary inequalities also correspond to a sign configuration of some twistor coordinates. Thus, we can test whether or not a point of the ambient Grassmannian is inside the amplituhedron solely by checking a finite number of sign configurations of its twistor coordinates. 
Third, these definitions are interesting from the point of view of positivity and cluster algebras. A recurring phenomenon for positive spaces is that they exhibit simpler features than expected. For example, to test if a plane of the Grassmannian is nonnegative it suffices to check the positivity of a strict subset of the set of Pl\"ucker coordinates, this is related to the cluster structure of the Grassmannian \cite{scott2006grassmannians}. This kind of behavior appears in the new definitions of the amplituhedron since they only depend on sign configurations of a strict subset of the set of twistor coordinates. One can see this new approach of the amplituhedron as a (small) step towards understanding the cluster structure of the amplituhedron.
Fourth, these three new definitions directly yield triangulations of the amplituhedron, see  \cite[Section 7]{arkani2018unwinding}. The sign flip triangulation is being studied: when $m=2$, it is proved in \cite{parisi2021m} that the sign flip triangulation is a so-called positroidal triangulation of the amplituhedron, and from the perspective of physics, these sign flip cells are used in \cite{kojima2020sign} to shed new lights on scattering amplitudes. While the sign flip triangulation has attracted most of the attention, we believe that the other two triangulations deserve a careful consideration.
Finally, the sign flip viewpoint contributes to the quest to unveil the yet-to-be discovered dual amplituhedron \cite{kojima2020sign, herrmann2021positive}.




Significant progresses have been made on the proof of the equivalence of the sign flip definition with the original one; in \cite{arkani2018unwinding} they sketched an argument to justify that if a point is in the amplituhedron then it satisfies the sign flip definition, in \cite{karp2019amplituhedron} they prove the same direction for $m=1$ and in \cite{parisi2021m} they prove the equivalence of the definitions for $m=2$. However no systematic study of the crossing and winding pictures has been made.


The goal of this paper is to prove one direction of the equivalence of the crossing and winding definitions with the original definition. More precisely, we prove that if a point of the ambient Grassmannian is inside the amplituhedron, then it has the correct winding or crossing number depending on the parity of $m$. Since it was already observed in \cite{arkani2018unwinding} that a point in the amplituhedron satisfies the boundary conditions, this indeed proves one way of the equivalence. In addition, we prove the complete equivalence of the original definition with the winding picture when $m=2$. Along the way, we also give a new proof that a point of the amplituhedron satisfies the sign flip definition.

A crucial tool in our argument is given by a new set of equations on twistor coordinates, valid for any $n,k,m$, that we called the $C$- and $\mathcal{Z}$- equations. When applied to the points of the amplituhedron, these equations together with the positivity conditions yield interesting constraints on sign patterns of twistor coordinates. For example, the sign flip definition directly follows from them. We hope that these equations will find applications beyond these proofs.



\subsection{The amplituhedron}

Let $n$ and $k$ be two nonnegative integers such that $0 \leq k \leq n$. 
\begin{defn}
	The (real) \emph{Grassmannian} ${\rm Gr}_{k,n}$ is the space of all $k$-dimensional vector spaces in $\mathbb{R}^n$.
\end{defn}  

We denote by $\left[n\right]$ the set $\left\{ 1,\dots,n\right\} $ and by $\sub k$ the set of lists of elements of $\left[n\right]$  of size $k$ sorted in ascending order. 
Let $I\in\sub k$ and $V\in\gr{k,n}$ represented by a $k\times n$ matrix $M$, then we denote by $p_{I}\left(V\right)$ the minor of $M$ relative to $I$. The minors $p_{I}\left(V\right)$ are the \emph{Pl\"ucker coordinates} of $V$; they do not depend on the choice of $M$ (up to simultaneous rescaling by a nonzero constant).

\begin{defn}
	The  \emph{totally nonnegative Grassmannian} ${\rm Gr}^{\geq 0}_{k,n}$ is the set of $k$-vector spaces $V \in \gr{k,n}$  such that $p_I(V) \geq 0$, for all $I \in {[n] \choose k}$. 
	
	The  \emph{totally positive Grassmannian} ${\rm Gr}^{\geq 0}_{k,n}$ is the set of $k$-vector spaces $V \in \gr{k,n}$  such that $p_I(V) > 0$, for all $I \in {[n] \choose k}$. 
	
	The space ${\rm Mat}_{k,n}^{ \geq 0}$ (resp. ${\rm Mat}_{k,n}^{> 0}$) is the set of $k \times n$ matrices such that $p_I(V) \geq 0$ (resp. $p_I(V) > 0$), for all $I \in {[n] \choose k}$.
\end{defn}
\begin{defn}[Amplituhedron]
	Let $\left(n,k,m\right)$ be a triplet of nonnegative integers such
	that $k+m\leq n$. Let $\mathcal{Z}$ be an element of ${\rm Mat}_{n,k+m}^{>0}$.
	This matrix induces a map 
	\begin{align*}
		\tilde{\mathcal{Z}}:\grp{k,n}            & \rightarrow\gr{k,k+m}                                                     \\
		{\rm Span}\left(c_{1},\dots,c_{k}\right) & \rightarrow{\rm Span}\left(\mathcal{Z}c_{1},\dots,\mathcal{Z}c_{k}\right) 
	\end{align*}
	or equivalently, if $C$ is a $k\times n$ matrix representing ${\rm Span}\left(c_{1},\dots,c_{k}\right)$
	in $\gr{k,n}^{\geq0}$ then $\tilde{\mathcal{Z}}\left(C\right)$ is
	 represented by $C\mathcal{Z}$. The \emph{(tree) amplituhedron}
	$\mathcal{A}_{n,k,m}$ is the image of $\gr{k,n}^{\geq0}$ by the
	map $\tilde{\mathcal{Z}}$. 
		
	We also denote by $\amplio$ the image of $\grsp{k,n}$ by $\tilde{\mathcal{Z}}$. 
\end{defn}

\begin{rem}
	The space $\amplio$ is a priori different from the interior of the amplituhedron. In \cite[Section~9]{galashin2020parity} the authors compare these two spaces, in particular they show that $\amplio$ is open and thus contained in the interior of the amplituhedron.
\end{rem}
\begin{defn}
	[Twistor coordinates]Let $\left(n,k,m\right)$ be a triplet of nonnegative
	integers such that $k+m\leq n$. Let $\mathcal{Z}\in{\rm Mat}_{n,k+m}^{>0}$
	and denote by $\mathcal{Z}_{1},\dots,\mathcal{Z}_{n}\in\mathbb{R}^{k+m}$
	its $n$ rows. Let $Y \in \gr{k,k+m}$ and denote by $Y_{1},\dots,Y_{k}\in\mathbb{R}^{k+m}$ the $k$ rows of a matrix representing $Y$. Let $\left(i_{1},\dots,i_{m}\right)$ be a list of
	elements of $\left[n\right]$. The \emph{twistor coordinate $\left(i_{1},\dots,i_{m}\right)$ of $Y$}, denoted by 
	\[
		\left\langle Y,\mathcal{Z}_{i_{1}},\dots,\mathcal{Z}_{i_{m}}\right\rangle, 
	\]
	is the determinant of $\left(Y_{1},\dots,Y_{k},\mathcal{Z}_{i_{1}},\dots,\mathcal{Z}_{i_{m}}\right)$. 
\end{defn}
\begin{rem}
	The twistor coordinates of $Y$ are defined up to a global nonzero factor,
	corresponding to the choice of a representative of the $k$-plane $Y$,
	but this factor will not matter. 
\end{rem}
\begin{notation}\label{Notation: twistor coordinates}When $\mathcal{Z}$
	is understood, we denote the twistor coordinate $\left\langle Y,\mathcal{Z}_{i_{1}},\dots,\mathcal{Z}_{i_{m}}\right\rangle $
	by $\left\langle Y,i_{1},\dots,i_{m}\right\rangle $. We will also
	denote the determinant of $\left(\mathcal{Z}_{i_{1}},\dots,\mathcal{Z}_{i_{k+m}}\right)$
	by $\left\langle i_{1},\dots,i_{k+m}\right\rangle $, for any list
	$\left(i_{1},\dots,i_{k+m}\right)$ of elements of $\left[n\right]$.
\end{notation}

\subsection{Coarse boundary of the amplituhedron}
\label{sec:Coarse_boundary}
Let $Y\in\ampli$. It follows from the Cauchy-Binet formula and the
positivity of the minors that certain twistor coordinates are positive:
when $m$ is even we have 
\begin{equation}
	\begin{cases}
		\left\langle Y,I\right\rangle \geq0                          & {\rm for\,\,}I=\left(i_{1},i_{1}+1,\dots,i_{\frac{m}{2}},i_{\frac{m}{2}}+1\right)\in{\left[n\right] \choose m},                 \\
		\left(-1\right)^{k+1}\left\langle Y,I,n,1\right\rangle \geq0 & {\rm for\,\,}I=\left(i_{1},i_{1}+1,\dots,i_{\frac{m}{2}-1},i_{\frac{m}{2}-1}+1\right)\in{\left[2,n-1\right] \choose m-2}, 
	\end{cases}\label{eq: coarse boundary m even}
\end{equation}
whereas when $m$ is odd we have
\begin{equation}
	\begin{cases}
		\left(-1\right)^{k}\left\langle Y,1,I\right\rangle \geq0 & {\rm for}\,\,I=\left(i_{1},i_{1}+1,\dots,i_{\frac{m+1}{2}-1},i_{\frac{m+1}{2}-1}+1\right)\in{\left[2,n\right] \choose m-1},   \\
		\left\langle Y,I,n\right\rangle \geq0                    & {\rm for}\,\,I=\left(i_{1},i_{1}+1,\dots,i_{\frac{m+1}{2}-1},i_{\frac{m+1}{2}-1}+1\right)\in{\left[n-1\right] \choose m-1}. 
	\end{cases}\label{eq: coarse boundary m odd}
\end{equation}
See Lemma~\ref{lem: Consequence Cauchy Binet} for a proof of these inequalities. Furthermore, these inequalities are strict if $Y\in\amplio$. In most of the cases, these inequalities are not refined enough to determine whether or not a point of $\gr{k,k+m}$ is in the amplituhedron.  
\begin{defn}
	We call Eq.~(\ref{eq: coarse boundary m even}) and Eq.~(\ref{eq: coarse boundary m odd})
	the \emph{coarse boundary conditions} of the amplituhedron.
\end{defn}
Let $\mathcal{L}$ be the locus of points in $\gr{k,k+m}$ where at least one of the inequalities of Eq.~(\ref{eq: coarse boundary m even}) or Eq.~(\ref{eq: coarse boundary m odd}),
depending on the parity of $m$, is an equality. We define
\[
	\ampliwb:=\ampli\backslash\mathcal{L}\quad{\rm and}\quad{\rm Gr}_{k,k+m}^{{\rm wcb}}:={\rm Gr}_{k,k+m}\backslash\mathcal{L},
\]
the notation referring to ``without coarse boundary''. In particular,
we have
\[
	\amplio\subset\ampliwb\subset\ampli.
\]

\subsection{Projection and simplices}

Fix a $k$-plane $Y$ in $\mathbb{R}^{k+m}$ and the $n$ vectors $\mathcal{Z}_{1},\dots,\mathcal{Z}_{n}$
in $\mathbb{R}^{k+m}$. Denote by $V_{Y}$ the quotient space $\mathbb{R}^{k+m}/Y$
and by
\[
	\pi_{Y}:\mathbb{R}^{k+m}\rightarrow V_{Y}\simeq\mathbb{R}^{m}
\]
the quotient map. 

\begin{notation} When $\mathcal{Z}$ and $Y$ are understood, we
	denote $Z_{i}:=\pi_{Y}\left(\mathcal{Z}_{i}\right)$ for $1\leq i\leq n$.
\end{notation} 
\begin{rem}
	\label{rem: det in projected space and twistors}Let $Y\in\gr{k,k+m}$
	and let $\left(f_{1},\dots,f_{m}\right)$ be a basis of a complement of $Y$ in $\mathbb{R}^{k+m}$, then $\mathcal{B}=\left(\pi_{Y}\left(f_{1}\right),\dots,\pi_{Y}\left(f_{m}\right)\right)$ is a basis of $V_{Y}$. Then, the list 
	\[
		\left(\det_{\mathcal{B}}\left(Z_{i_{1}},\dots,Z_{i_{m}}\right),1\leq i_{1},\dots,i_{m}\leq n\right),
	\]
	where $\det_{\mathcal{B}}$ is the determinant in the basis $\mathcal{B}$, is equal to the list 
	\[
		\left(\left\langle Y,\mathcal{Z}_{i_{1}},\dots,\mathcal{Z}_{i_{m}}\right\rangle ,1\leq i_{1},\dots,i_{m}\leq n\right),
	\]
	up to a global nonzero factor. Thus, we interpret twistor coordinates as determinants of $Z_{i}$ in $V_{Y}$. 
\end{rem}
The definitions of the winding number and the crossing number use
simplices in $V_{Y}$ that we introduce now. 
\begin{defn}
	For any subset $I$ of $\left[n\right]$, we denote $\mathcal{S}\left(I\right)$
	the simplex in $\mathbb{R}^{k+m}$ given by the convex hull of the points
	$\mathcal{Z}_{i}$, for $i\in I$ in $\mathbb{R}^{k+m}$. We denote
	by $S\left(I\right)$ the image of the simplex $\mathcal{S}\left(I\right)$
	by $\pi_{Y}$. We call $S\left(I\right)$ a simplex even if its vertices
	are not necessarily affinely independent. If the vertices $Z_{i}$,
	for $i\in I$, of $S\left(I\right)$ are affinely independent, we
	say that $S\left(I\right)$ is \emph{full-dimensional. }
\end{defn}

\subsection{The winding number}

Let $m$ be an even positive integer. Fix $k,n$ such that $n\geq k+m$.
Fix also $\mathcal{Z}\in{\rm Mat}_{n,k+m}^{>0}$, and denote by $\mathcal{Z}_{1},\dots,\mathcal{Z}_{n}$
its $n$ rows.

We associate a number, the winding number, to any point $Y\in{\rm Gr}_{k,k+m}^{{\rm wcb}}$
in the following way. Let $P\left(Y,\mathcal{Z}\right)$ be the following
polyhedron in $V_{Y}$
\[
	P\left(Y,\mathcal{Z}\right):=\bigcup_{I}S\left(I\right)\bigcup_{J}S\left(J,n,1\right)\subset V_{Y},
\]
where $I=\left(i_{1},i_{1}+1,\dots , i_{\frac{m}{2}},i_{\frac{m}{2}}+1\right)$
is a strictly ascending list of positive integers between $1$ and
$n$, and $J=\left(j_{1},j_{1}+1,\dots,j_{\frac{m-1}{2}},j_{\frac{m-1}{2}}+1\right)$
is a strictly ascending list of positive integers between $2$ and
$n-1$. Let $S^{k-1}$ be the $\left(k-1\right)$-sphere centered
in the origin of $V_{Y}$ of radius $1.$ We define the map
\begin{align*}
	f: P\left(Y,\mathcal{Z}\right) & \rightarrow S^{k-1}                  \\
	x   & \rightarrow\frac{x}{\lVert x\rVert}, 
\end{align*}
where $\lVert\cdot\rVert$ is a norm on $V_{Y}$. Note
that since $Y\in{\rm Gr}_{k,k+m}^{{\rm wcb}}$, the origin of $V_{Y}$
does not belong to $P\left(Y,\mathcal{Z}\right)$ and then $f$ is
well defined.
\begin{defn}
	[Winding number] The \emph{winding number}
	\[
		w_{n,k,m}\left(Y,\mathcal{Z}\right)
	\]
	is the degree of the map $f$ (i.e. the image of the homology class
	of $P\left(Y,\mathcal{Z}\right)$ by the map $f_{*}:H_{k-1}\left(P\left(Y,\mathcal{Z}\right)\right)\rightarrow H_{k-1}\left(S^{k-1}\right)\simeq\mathbb{Z}$). 
		
	It corresponds to the number of times that $P\left(Y,\mathcal{Z}\right)$ winds around the origin
	of $V_{Y}$.
\end{defn}
\begin{rem}
	For a generic $Y \in \gr{k,k+m}$, the ray issued from almost all vectors of $V_{Y}$ intersects the simplices of $P\left(Y,\mathcal{Z}\right)$ only in their $m$-dimensional interior. Pick $Y$ generic enough and such a vector, say $Z_{*}$, in $V_{Y}$. Then the degree of $f$ is given by the number of simplices intersecting the ray issued from $Z_{*}$ weighted by a sign corresponding to the orientation of the simplex. We recover in these cases the definition of the winding number of \cite[Section 3]{arkani2018unwinding}.
\end{rem}
\begin{thm}
	\label{thm: winding constant}The winding is independent of $Y\in\ampliwb$,
	of $\mathcal{Z}\in\matp$, and it equals
	\begin{equation}
		w_{n,k,m}\left(Y,\mathcal{Z}\right)={\left\lfloor \frac{k+m-1}{2}\right\rfloor  \choose \frac{m}{2}}.\label{eq: explicit winding number intro}
	\end{equation}
\end{thm}
\begin{rem}
	It is apparent from Eq.~(\ref{eq: explicit winding number intro}) that the winding number is also independent of $n$.
\end{rem}

A point in $\ampliwb$ also satisfies the coarse boundary conditions Eq.~(\ref{eq: coarse boundary m even}). In \cite[Section 3]{arkani2018unwinding}, the authors also conjectured the inverse implication: a point with the correct winding number and satisfying the coarse boundary condition is in the amplituhedron.  This implication is still open.

\begin{open} Does a point $Y\in\grwcb{k,k+m}$ satisfying the inequalities of Eq.~(\ref{eq: coarse boundary m even}) and satisfying Eq.~(\ref{eq: explicit winding number intro}) belong to $\ampli$?\end{open}

When $m=2$, the following proposition gives a positive answer to this problem. It allows to give an alternative definition of the $m=2$ amplituhedron in terms of the winding number, see Corollary \ref{m=2 def}.

\begin{prop}
	\label{prop: maximality}
	Fix $m=2$. Let $Y\in\grwcb{k,k+2}$ and $\mathcal{Z}\in{\rm Mat}_{n,k+2}^{>0}$. 
	\begin{enumerate}
		\item We have
		      \[
		      	w_{n,k,2}\left(Y,\mathcal{Z}\right)\leq\left\lfloor \frac{k+1}{2}\right\rfloor .
		      \]
		\item If $w_{n,k,2}\left(Y,\mathcal{Z}\right)=\left\lfloor \frac{k+1}{2}\right\rfloor $
		      and satisfies the coarse boundary conditions, see Eq.~(\ref{eq: coarse boundary m even}), then $Y\in\mathcal{A}_{n,k,2}$. 
	\end{enumerate}
\end{prop}
\begin{cor}
	\label{m=2 def}
	Let $Y\in\grwcb{k,k+2}$ and $\mathcal{Z}\in{\rm Mat}_{n,k+2}^{>0}$.
	The point $Y$ is in $\mathcal{A}_{n,k,2}$ if and only if $w_{n,k,2}\left(Y,\mathcal{Z}\right)=\left\lfloor \frac{k+1}{2}\right\rfloor $
	and $Y$ satisfies the coarse boundary conditions, see  Eq.~(\ref{eq: coarse boundary m even}). 
\end{cor}

\subsection{The crossing number\label{subsec: the crossing number}}

Let $m$ be an odd positive integer, we write $m=2r-1$ with $r$
a positive integer. Fix $k,n$ such that $n\geq k+m$. Fix also $\mathcal{Z}\in{\rm Mat}_{n,k+m}^{>0}$, and denote by $\mathcal{Z}_{1},\dots,\mathcal{Z}_{n}$ its $n$ rows.

We associate a number, the crossing number, to any point $Y\in\gr{k,k+m}$
in the following way. Start from a $k$-plane $Y$ in $\mathbb{R}^{k+m}$
and the $n$ vectors $\mathcal{Z}_{1},\dots,\mathcal{Z}_{n}$ in $\mathbb{R}^{k+m}$.
Let $\left(i_{1},i_{1}+1,\dots,i_r,i_{r}+1\right)$ be strictly ascending  list of $m+1$ positive integers between $1$ and $n$. The list is composed of $r$ pairs of consecutive integers. We decompose
the simplex $S\left(i_{1},\dots,i_{r}+1\right)$ into cells:
\begin{itemize}
	\item the \emph{$0$-cells} are the vertices $Z_{i_{1}},Z_{i_{1}+1}\dots,Z_{i_{r}}$
	      and $Z_{i_{r}+1}$ of $S\left(i_{1},\dots,i_{r}+1\right)$, 
	\item let $1\leq a\leq2r-1$. For any choice of $\left(a+1\right)$ affinely
	      independent vertices of $S\left(i_{1},\dots,i_{r}+1\right)$, say
	      $Z_{j_{1}},\dots,Z_{j_{a+1}}$ where $\left\{ j_{1},\dots,j_{a+1}\right\} \subset\left\{ i_{1},i_{1}+1,i_{2}+1,\dots,i_{r}+1\right\} $,
	      we define the \emph{$a$-cell} $\ensuremath{\mathfrak{C}}\left(j_{1},\dots,j_{a+1}\right)$
	      to be the relative interior of the convex hull of $\pi_{Y}\left(\mathcal{Z}_{j_{1}}\right),\dots,\pi_{Y}\left(\mathcal{Z}_{j_{a+1}}\right)$. 
\end{itemize}
Denote by $\mathfrak{Cells}$ the set of all the cells of the simplices
$S\left(i_{1},\dots,i_{r}+1\right)$ for any list \\ 
$\left(i_{1},i_{1}+1,i_{2},\dots,i_{r}+1\right)$
of strictly ascending integers between $1$ and $n$. 
\begin{defn}
	\label{def: crossing number}The \emph{crossing number 
		\[
			c_{n,k,m}\left(Y,\mathcal{Z}\right)
		\]
		} is the number of cells in $\mathfrak{Cells}$ containing the origin
	$\pi_{Y}\left(Y\right)$ of $V_{Y}$. 
\end{defn}
\begin{rem}
	\label{rem: crossing as altenating signs}Suppose the origin $\pi_{Y}\left(Y\right)$ of $V_{Y}$ is contained only in $m$-dimensional cells (i.e. in the interior of the simplices). Then the crossing number counts the number of simplices of type $S\left(i_{1},i_{1}+1,i_{2},\dots,i_{r}+1\right)$ containing the origin $V_{Y}$. However, the sign of $\left\langle Y,I\backslash\left\{ i\right\} \right\rangle $ determines the relative position of the origin with respect to the hyperplane containing the affine hull of $\left\{ Z_{j},j\in I\backslash\left\{ i\right\} \right\} $. Thus, in this particular case, the origin belongs to the simplex $S\left(I\right)$ if and only if the sequence
	\[
		\left({\rm sign}\left\langle Y,I\backslash\left\{ i\right\} \right\rangle \right)_{i\in I}
	\]
	is alternating. Thus, Definition~\ref{def: crossing number} of
	the crossing number agrees with the definition of \cite[Section 4]{arkani2018unwinding}.
	Otherwise, the two definitions differ and Theorem~\ref{thm: crossing is constant}
	only holds using Definition~\ref{def: crossing number}. 
\end{rem}
\begin{thm}
	\label{thm: crossing is constant}The crossing number is independent
	of $Y\in\mathcal{A}_{n,k,m}^{>0}$ and of $\mathcal{Z}\in{\rm Mat}_{n,k+m}^{>0}$, and it equals
	\begin{equation}
		c_{n,k,m}\left(Y,\mathcal{Z}\right)=\begin{cases}
		\frac{2k+m-1}{m+1}{\frac{k+m-2}{2} \choose \frac{m-1}{2}} & {\rm for\;}k\;{\rm odd},\\
		2{\frac{k+m-1}{2} \choose \frac{m+1}{2}} & {\rm for\;}k\;{\rm even.}
		\end{cases}\label{eq: explicit crossing number}
	\end{equation}
\end{thm}
\begin{rem}
	It is apparent from Eq.~(\ref{eq: explicit crossing number})
	that the crossing number is independent of $n$.
\end{rem}
\begin{rem}
	From the explicit formulas of the winding number Eq.~(\ref{eq: explicit winding number intro})
	and of the crossing number Eq.~(\ref{eq: explicit crossing number}),
	we obtain the following relation between the crossing and the winding
	numbers 
	\[
		c_{n,k,m}\left(Y,\mathcal{Z}\right)=\begin{cases}
		2w_{n,k,m+1}\left(Y,\mathcal{Z}\right)-w_{n,k,m-1}\left(Y,\mathcal{Z}\right) & {\rm for\;}k\;{\rm odd},\\
		2w_{n,k,m+1}\left(Y,\mathcal{Z}\right) & {\rm for\;}k\;{\rm even},
		\end{cases}
	\]
	for $\left(Y,\mathcal{Z}\right) \in \amplio \times \matp$. This relation can be directly deduced by geometrical reasons. See \cite[Section 4]{arkani2018unwinding}
	for an idea of the argument. 
\end{rem}
A point in $\amplio$ also satisfies the coarse boundary conditions Eq.~(\ref{eq: coarse boundary m odd}). In \cite[Section 4]{arkani2018unwinding}, the authors also conjectured the inverse implication: a point with the correct crossing number and satisfying the coarse boundary condition is in the amplituhedron.  This implication is still open.

\begin{open} Does a point $Y\in\gr{k,k+m}$ satisfying Eq.~(\ref{eq: coarse boundary m odd})
	and Eq.~(\ref{eq: explicit crossing number}) belong to $\ampli$?
\end{open}

\subsection{Organization of the paper}

In Section~\ref{sec: Winding}, we prove Theorem~\ref{thm: winding constant}.
We first prove that the winding number $w_{n,k,m}\left(Y,\mathcal{Z}\right)$
in constant when $\left(Y,\mathcal{Z}\right)\in\ampliwb\times\matp$.
Then we prove that the winding number is also independent of $n$.
Finally, we obtain the explicit expression of the winding number for
$n=k+m$. In the last subsection, we prove Proposition \ref{prop: maximality} establishing the equivalence between the winding description and the original description of the amplituhedron when $m=2$. 

In section~\ref{sec: crossing number}, we prove Theorem~\ref{thm: crossing is constant}.
The proof follows the same path. We first prove that the crossing
number $c_{n,k,m}\left(Y,\mathcal{Z}\right)$ is constant when $\left(Y,\mathcal{Z}\right)\in\amplio\times\matp$.
In this case, this part is more subtle. It requires to first prove
two sets of equations on twistor coordinates that we call $C$-equations
and $\mathcal{Z}$-equations. We emphasize that these equations are
valid for every $n,k$ and $m$. We deduce from these equations that
we can avoid unpleasant behavior of the simplices and cells involved
in the count of the crossing number, and then prove the constancy
of the crossing number in $\amplio\times\matp$. Then we prove the
independence of the crossing number in $n$, and finally we obtain
the explicit expression of the crossing for $n=k+m$.

The two sections are independent.

\subsection{Acknowledgments}

We are grateful to Ran Tessler for valuable discussions and comments.
X. B. was supported by the ISF grant 335/19 in the group of Ran Tessler. 
J.L. was supported by the Austrian Science Fund (FWF): P 34602 individual project. 

\section{The winding number\label{sec: Winding}}

In this section $m$ will denote an even integer. 

\subsection{Preliminaries}

Throughout the text, we use the Cauchy-Binet formula to develop twistor
coordinates. A formulation of this development is given by Lemma 3.3
of \cite{parisi2021m}. Let us recall the statement.
\begin{lem}
	[\cite{parisi2021m}]\label{lem: Consequence Cauchy Binet}Let
	$\mathcal{Z}\in\matp$. Write $Y\in{\rm Gr}_{k,k+m}$ as $Y=C\mathcal{Z}$
	with $C\in{\rm Gr}_{k,n}$. We have 
	\[
		\left\langle C\mathcal{Z},i_{1},\dots,i_{m}\right\rangle =\sum_{J=\left(j_{1}<\cdots<j_{k}\right)\in\sub k}p_{J}\left(C\right)\left\langle j_{1},\dots,j_{k},i_{1},\dots,i_{m}\right\rangle ,
	\]
	where we used Notation~\ref{Notation: twistor coordinates}.
\end{lem}
For completeness, we include a proof. 
\begin{proof}
	We have
	\[
		\left\langle C\mathcal{Z},i_{1},\dots,i_{m}\right\rangle =\det\left(\left(\begin{array}{c}
		C\\
		I_{i_{1},\dots,i_{m}}
		\end{array}\right)\mathcal{Z}\right),
	\]
	where $I_{i_{1},\dots,i_{m}}$ is the $m\times n$ matrix such that
	the $l$th row has a $1$ in position $i_{l}$ and zeros elsewhere.
	We then use the Cauchy-Binet formula to obtain the result. 
\end{proof}
In particular, (strict) inequalities of the coarse boundary conditions,
Eq.~(\ref{eq: coarse boundary m even}) and Eq.~(\ref{eq: coarse boundary m odd}),
follow from the positivity of $\mathcal{Z}\in\matp$ and the (strict)
positivity of $C\in\grp{k,n}$.

\subsection{Constancy of the winding number in $\ampliwb \times \matp$}
\begin{prop}
	\label{prop:winding_independent_CZ}
	Fix $n,k$ and $m$ even such that $n\geq k+m$. The winding number
	$w_{n,k,m}:\ampliwb\times\matp\rightarrow\mathbb{Z}$ is constant.
\end{prop}
\begin{rem}
	\label{rem:constancy-winding-Gr}
	We can prove in a similar way that the winding number is constant on each path-connected component of $\grwcb{k,k+m}\times\matp$. 
\end{rem}
\begin{proof}
	Let $\mathcal{Z},\mathcal{Z}^{'}\in\matp$ and $Y,Y^{'}\in\ampliwb$.
	Let $C,C^{'}\in\grp{k,n}$ such that
	\[
		Y=C\mathcal{Z}\quad{\rm and}\quad Y^{'}=C^{'}\mathcal{Z}^{'}.
	\]
	Since $\grsp{k,n}$, $\matp$ are path connected (see \cite{postnikov2006total} for the positive Grassmannian) and $\grp{k,n}=\overline{{\rm Gr}_{k,n}^{>0}}$, there exists a path
	\[
		\left(\breve{C},\breve{\mathcal{Z}}\right):\left[0,1\right]\rightarrow\grp{k,n}\times\matp
	\]
	such that $\left(\breve{C}\times\breve{\mathcal{Z}}\right)\left(0\right)=\left(C,\mathcal{Z} \right)$,
	$\left(\breve{C}\times\breve{\mathcal{Z}}\right)\left(1\right)=\left(C^{'},\mathcal{Z}^{'}\right)$
	and such that $\breve{C}\left(t\right)\in\grsp{k,n}$ for $t\in\left]0,1\right[$.
	The winding number 
	\[
		w_{n,k,m}\left(\breve{Y}\left(t\right),\breve{\mathcal{Z}}\left(t\right)\right),
	\]
	where $\breve{Y}\left(t\right):=\breve{C}\breve{\mathcal{Z}}$, is
	independent of $t$. Indeed, the winding number may change only as
	a result of an intersection between the origin $\pi_{\breve{Y}\left(t\right)}\left(\breve{Y}\left(t\right)\right)$
	of $V_{Y}$ and the polytope $P\left(\breve{Y}\left(t\right),\breve{\mathcal{Z}}\left(t\right)\right)$.
	However, for every $t\in\left[0,1\right]$ we have $\breve{Y}\left(t\right)\in\ampliwb$, hence $\breve{Y}\left(t\right)$ satisfies the strict inequality of the coarse boundary condition Eq.~(\ref{eq: coarse boundary m even}), thus the origin $\pi_{\breve{Y}\left(t\right)}\left(\breve{Y}\left(t\right)\right)$ of $V_{\breve{Y}\left(t\right)}$ does not hit the polytope $P\left(\breve{Y}\left(t\right),\breve{\mathcal{Z}}\left(t\right)\right)$.
\end{proof}

\subsection{Independence of the winding number in $n$}

We show in this section that for a specific choice of points in $\ampliwb$
and $\matp$, the winding number does not depend on $n$. Since, we
proved that the winding number $w_{n,k,m}\left(Y,\mathcal{Z}\right)$
is independent of $Y\in\ampliwb$ and $\mathcal{Z}\in\matp$, we deduce
that the winding number is also independent of $n$. 
\begin{prop}
	\label{prop:winding_independent_n}
	There exist $\left(C^{'},\mathcal{Z}^{'}\right)\in {\rm Gr}_{k,n+1}^{>0} \times {\rm Mat}_{n+1,k+m}^{>0}$
	and $\left(C,\mathcal{Z}\right) \in \grsp{k,n}\times\matp$ such that
	\begin{equation}
		w_{n,k,m}\left(C\mathcal{Z},\mathcal{Z}\right)=w_{n+1,k,m}\left(C^{'}\mathcal{Z}^{'},\mathcal{Z}^{'}\right).\label{eq: winding indep n}
	\end{equation}
\end{prop}
\begin{proof}
	In Step~$1$, we introduce $\left(C^{'},\mathcal{Z}^{'}\right)$
	and $\left(C,\mathcal{Z}\right)$. Then, in Step~$2$ we view the
	polytopes $P\left(C\mathcal{Z},\mathcal{Z}\right)$ and $P\left(C^{'}\mathcal{Z}^{'},\mathcal{Z}^{'}\right)$
	associated to these points as singular chains. We prove that the difference
	$Q=P\left(C^{'}\mathcal{Z}^{'},\mathcal{Z}^{'}\right)-P\left(C\mathcal{Z},\mathcal{Z}\right)$
	is given by a sum of boundaries of singular $m$-simplices. Finally
	in Step~$3$, we use this formula to prove that the winding of $Q$
	is zero. This implies that the winding of $P\left(C^{'}\mathcal{Z}^{'},\mathcal{Z}^{'}\right)$
	is equal to the winding of $P\left(C\mathcal{Z},\mathcal{Z}\right)$. 
		
	\paragraph{Step~$1$.}
		
	Choose $\mathcal{Z}^{'}=\left(\mathcal{Z}_{1},\dots,\mathcal{Z}_{n+1}\right)\in{\rm Mat}_{n+1,k+m}^{>0}$
	 and then define $\mathcal{Z}=\left(\mathcal{Z}_{1},\dots,\mathcal{Z}_{n}\right)\in\matp$.
	Let $C\in\grsp{k,n}$ and let $Y=C\mathcal{Z}\in\amplio$. We define
	the matrix $C^{'}$ by adding a $\left(n+1\right)$th column of zeros
	to $C$. It follows from Lemma~\ref{lem: Consequence Cauchy Binet}
	together with the positivity of $\mathcal{Z}^{'}$ and $C$ that $Y=C^{'}\mathcal{Z}^{'}$
	does not belong to the coarse boundary of the amplituhedron: 
	\begin{equation}
		\left\langle C^{'}\mathcal{Z}^{'},I\right\rangle \neq0\label{eq: non vanishing boundary twistor}
	\end{equation}
	for any subset $I=\left(i_{1},i_{1}+1,\dots,i_{\frac{m}{2}},i_{\frac{m}{2}}+1\right)$
	or $I=\left(i_{1},i_{1}+1,\dots,i_{\frac{m}{2}-1},i_{\frac{m}{2}-1}+1,n+1,1\right)$
	in ${\left[n+1\right] \choose m}$. Hence the winding number $w_{n+1,k,m}\left(Y,\left(\mathcal{Z}_{1},\dots,\mathcal{Z}_{n+1}\right)\right)$
	is well-defined.
		
	\paragraph{Step $2$. }
		
	Let $P\left(Y,\mathcal{Z}\right)$ and $P\left(Y,\mathcal{Z}^{'}\right)$
	be the polytopes associated to $\mathcal{Z}$ and $\mathcal{Z}^{'}$
	in $V_{Y}$. We consider $P\left(Y,\mathcal{Z}\right)$ as a singular
	chain given by the sum of the singular simplices
	\[
		\Delta^{m-1}\rightarrow S\left(i_{1},i_{1}+1,\dots ,i_{\frac{m}{2}}+1\right)\quad{\rm and}\quad\Delta^{m-1}\rightarrow S\left(i_{1},i_{1}+1, \dots,i_{\frac{m}{2}-1}+1,n,1\right)
	\]
	where $\Delta^{m-1}$ is the standard $\left(m-1\right)$-simplex
	$\left(e_{1},\dots,e_{m}\right)$ in $\mathbb{R}^{m}$ and the first
	(resp. second) map is the linear map sending the basis $\left(e_{1},\dots,e_{m}\right)$
	of $\mathbb{R}^{m}$ to $\left(Z_{i_{1}},Z_{i_{1}+1},\dots,Z_{i_{\frac{m}{2}}+1}\right)$
	(resp. $\left(Z_{i_{1}},Z_{i_{1}+1},\dots,Z_{n},Z_{1}\right)$). Similarly,
	we consider $P\left(Y,\mathcal{Z}^{'}\right)$ as a singular chain.
	We have 
	\[
		\partial P\left(Y,\mathcal{Z}\right)=\partial P\left(Y,\mathcal{Z}^{'}\right)=0.
	\]
	Write $Q:=P\left(Y,\mathcal{Z}^{'}\right)-P\left(Y,\mathcal{Z}\right)$.
	It is then a closed chain and does not intersect the origin, so its
	winding number is well defined. Moreover, the winding number associated
	to $P\left(Y,\mathcal{Z}^{'}\right)$ is equal to the winding number
	of $P\left(Y,\mathcal{Z}\right)$ plus the winding number of $Q$.
	We will prove that the winding number of $Q$ is zero, hence proving
	Eq.~(\ref{eq: winding indep n}).
		
	We first prove that 
	\begin{equation}
		Q=\sum\partial\sigma_{i_{1},i_{1}+1,\dots,i_{\frac{m}{2}-1},i_{\frac{m}{2}-1}+1,n,n+1,1},\label{eq: the chain Q}
	\end{equation} 
	where the summation is over strictly ascending lists of integers $\left(i_{1},i_{1}+1,\dots,i_{\frac{m}{2}-1},i_{\frac{m}{2}-1}+1\right)$ between $2$ and $n-1$. We used the notation $\sigma\left(I\right)$ to denote the singular simplex $\Delta^{\left|I\right|-1}\rightarrow S\left(I\right)$. 
		
	The singular simplices $\sigma_{i_{1},i_{1}+1,\ldots,i_{\frac{m}{2}-1},i_{\frac{m}{2}-1}+1,n,n+1}$ and $\sigma_{i_{1},i_{1}+1,\ldots,i_{\frac{m}{2}-1},i_{\frac{m}{2}-1}+1,n+1,1}$ in \\
	$\partial\sigma_{i_{1},i_{1}+1,\dots,i_{\frac{m}{2}-1},i_{\frac{m}{2}-1}+1,n,n+1,1}$ come with a positive sign. These simplices belong to the set of simplices
	of $P\left(Y,\mathcal{Z}^{'}\right)$ not in $P\left(Y,\mathcal{Z}\right)$. Since $i_1 \geq 2$ and $i_{\frac{m}{2}+1}\leq n-1$ they do not constitute the whole set of simplices in $P\left(Y,\mathcal{Z}^{'}\right)$ and not in $P\left(Y,\mathcal{Z}\right)$. On the other hand, the simplices $\sigma_{i_{1},i_{1}+1,\ldots,i_{\frac{m}{2}-1},i_{\frac{m}{2}-1}+1,n,1}$
	come with a minus sign and are exactly the simplices of $P\left(Y,\mathcal{Z}\right)$
	not in $P\left(Y,\mathcal{Z}^{'}\right)$. 
		
	The rest of the simplices are of the form $\sigma_{i_{1},i_{1}+1,\ldots,\widehat{i_{j}+\epsilon},\dots,i_{\frac{m}{2}-1},i_{\frac{m}{2}-1}+1,n,n+1,1}$
	for some $1\leq j \leq\frac{m}{2}-1$ and $\epsilon\in\{0,1\}.$ The
	set of indices $\left\{ i_{1},i_{1}+1,\ldots,\widehat{i_{j}+\epsilon},\dots,i_{\frac{m}{2}-1},i_{\frac{m}{2}-1}+1\right\} $
	is uniquely described as disjoint union of intervals $I_{1}\sqcup I_{2}\sqcup\ldots\sqcup I_{r}$
	where each $I_{k}$ is a sequence of consecutive integers in $\{2,\ldots,n-1\}$,
	and the least element of $I_{k+1}$ is greater by at least $2$ from
	the largest element of $I_{k}.$ Clearly, there is a unique
	interval of odd size, we denote it by $I_{k_{0}}$. There are now
	two possibilities; either $\left\{ 2\right\} \notin I_{k_{0}}$ and $\left\{ n-1\right\} \notin I_{k_{0}}$,
	or exactly one element of $\left\{ 2,n-1\right\} $ belongs to $I_{k_{0}}$
	(if both of them were in $I_{k_{0}}$ then $m$ must have been $0$).
		
	If $2,n-1\notin I_{k_{0}},$ then each simplex $\sigma_{i_{1},i_{1}+1,\ldots,\widehat{i_{j}+\epsilon},\dots,i_{\frac{m}{2}-1},i_{\frac{m}{2}-1}+1,n,n+1,1}$
	appears twice in $Q$ with opposite sign. Indeed, let $I_{k_{0}}^{+}$
	be the extension of $I_{k_{0}}$ by the least integer greater
	than all the elements of $I_{k_{0}},$ and let $I_{k_{0}}^{-}$ be
	the extension of $I_{k_{0}}$ by the largest integer smaller than
	all the elements of $I_{k_{0}}.$ Write
	\[
		I^{\pm}=I_{k_{0}}^{\pm}\sqcup\bigsqcup_{i\neq k_{0}}I_{i}.
	\]
	Then $\sigma_{i_{1},i_{1}+1,\ldots,\widehat{i_{j}+\epsilon},\dots,i_{\frac{m}{2}-1},i_{\frac{m}{2}-1}+1,n,n+1,1}$
	appears exactly in $\partial\sigma_{I^{+},n,n+1,1}$ and $\partial\sigma_{I^{-},n,n+1,1}$
	with opposite sign.
		
	Now suppose that $2$ or $n-1$ belongs to $I_{k_{0}}$. Suppose first that  $2$ belongs to $I_{k_{0}}$, then $I_{k_{0}}=I_{1}$
	and $\epsilon=1$. Then the simplex $\sigma_{i_{1},i_{1}+1,\ldots,\widehat{i_{j}+\epsilon},\dots,i_{\frac{m}{2}-1},i_{\frac{m}{2}-1}+1,n,n+1,1}$
	is the simplex of $P\left(Y,\mathcal{Z}^{'}\right)$ with the set
	of indices
	$\left\{ j_{1},j_{1}+1,\dots,j_{\frac{m}{2}-1},j_{\frac{m}{2}-1}+1,n,n+1\right\} $
	such that 
	\[
		j_{1}=1\,\,{\rm and}\,\,I_{1}=\left\{ j_{1}+1,j_{2},j_{2}+1, \dots ,j_{l_{1}}+1\right\} ,\dots,I_{r}=\left\{ j_{l_{r-1}},j_{l_{r-1}}+1,\dots,j_{\frac{m}{2}-1},j_{\frac{m}{2}-1}+1\right\} .
	\]
	Moreover, this simplex $\sigma_{j_{1},j_{1}+1,\dots,j_{\frac{m}{2}-1},j_{\frac{m}{2}-1}+1,n,n+1}$ appears in the sum of Eq.~(\ref{eq: the chain Q}) with a positive sign, thus it contributes to the simplices of $P\left(Y,\mathcal{Z}^{'}\right)$ not in $P\left(Y,\mathcal{Z}\right)$. In the same way, if $n-1\in I_{k_{0}}$, we obtain simplices of type $\sigma_{i_{1},i_{1}+1,\dots,i_{\frac{m}{2}},i_{\frac{m}{2}}+1,n+1,1}$, where $i_{\frac{m}{2}}+1=n$ contributing with a positive sign to Eq.~(\ref{eq: the chain Q}). This completes the list of simplices of
	$P\left(Y,\mathcal{Z}^{'}\right)$ not in $P\left(Y,\mathcal{Z}\right)$. 
		
	Putting everything together, we see that all simplices appearing
	in Eq.~(\ref{eq: the chain Q}) cancel in pairs, except exactly those
	simplices which are simplices $P\left(Y,\mathcal{Z}^{'}\right)$ but
	not of $P\left(Y,\mathcal{Z}\right)$, or of $P\left(Y,\mathcal{Z}\right)$
	but not of $P\left(Y,\mathcal{Z}^{'}\right)$, and in both of these
	cases they appear with the correct sign. Thus, Eq.~(\ref{eq: the chain Q})
	holds. 
		
	\paragraph{Step $3$.}
		
	We now show that the winding number of $Q$ is zero. Let \\$\left(i_{1},i_{1}+1,\dots,i_{\frac{m}{2}-1},i_{\frac{m}{2}-1}+1\right)$
	be a strictly ascending list of integers between $2$ and $n-2$.
	The boundary \\
	$\partial\sigma_{i_{1},i_{1}+1,\dots,i_{\frac{m}{2}-1},i_{\frac{m}{2}-1}+1,n,n+1,1}$
	is a cycle and this cycle does not touch the origin of $V_{Y}$. Indeed,
	up to modifying $C$ in $\grsp{k,n}$, we have 
	\[
		\left\langle C^{'}\mathcal{Z}^{'},i_{1},i_{1}+1,\dots,\widehat{i_{j}+\epsilon},\dots,i_{\frac{m}{2}-1},i_{\frac{m}{2}-1}+1,n,n+1,1\right\rangle \neq0,
	\]
	for $j\in\left[\frac{m}{2}\right]$ and $\epsilon\in\left\{ 0,1\right\} $.
	This is possible since by Lemma~\ref{lem: Consequence Cauchy Binet}
	these twistor coordinates only involve the minors of $C$, and the
	vanishing of these twistor coordinates is a codimension $1$ locus.
	Moreover, the rest of the twistor coordinates $\left\langle Y,i_{1},i_{1}+1,\dots,i_{\frac{m}{2}-1},i_{\frac{m}{2}-1}+1,\widehat{n},n+1,1\right\rangle $,
	$\left\langle Y,i_{1},i_{1}+1,\dots,i_{\frac{m}{2}-1},i_{\frac{m}{2}-1}+1,n,\widehat{n+1},1\right\rangle $
	and $\left\langle Y,i_{1},i_{1}+1,\dots,i_{\frac{m}{2}-1},i_{\frac{m}{2}-1}+1,n,n+1,\widehat{1}\right\rangle $
	do not vanish since $Y=C\mathcal{Z}=C^{'}\mathcal{Z}^{'}$ does not
	belong to the coarse boundary of the amplituhedron. Hence the winding
	number of $\partial\sigma_{i_{1},i_{1}+1,\dots,i_{\frac{m}{2}-1},i_{\frac{m}{2}-1}+1,n,n+1,1}$
	is well defined. Moreover, it follows from Eq.~(\ref{eq: the chain Q})
	that the winding number of $Q$ is the sum of the winding numbers
	of such cycles. We show that the winding number of $\partial\sigma_{i_{1},i_{1}+1,\dots,i_{\frac{m}{2}-1},i_{\frac{m}{2}-1}+1,n,n+1,1}$
	is zero.
		
	Suppose the winding number of $\partial\sigma_{i_{1},i_{1}+1,\dots,i_{\frac{m}{2}-1},i_{\frac{m}{2}-1}+1,n,n+1,1}$
	is nonzero, then the simplex $\sigma_{i_{1},i_{1}+1,\dots,i_{\frac{m}{2}-1},i_{\frac{m}{2}-1}+1,n,n+1,1}$
	is of dimension $m$ and contains the origin of $V_{Y}$. As mentioned
	in Remark~\ref{rem: crossing as altenating signs}, this implies
	that the sign of its twistor coordinates is alternating:
	\begin{align*}
{\rm sgn}\left\langle \widehat{i_{1}},i_{1}+1,\dots,i_{\frac{m}{2}-1},i_{\frac{m}{2}-1}+1,\right. & \left.n,n+1,1\right\rangle \\
 & =-{\rm sgn}\left\langle i_{1},\widehat{i_{1}+1},\dots,i_{\frac{m}{2}-1},i_{\frac{m}{2}-1}+1,n,n+1,1\right\rangle \\
 & =\cdots\\
 & ={\rm sgn}\left\langle i_{1},i_{1}+1,\dots,i_{\frac{m}{2}-1},i_{\frac{m}{2}-1}+1,\widehat{n},n+1,1\right\rangle \\
 & =-{\rm sgn}\left\langle i_{1},i_{1}+1,\dots,i_{\frac{m}{2}-1},i_{\frac{m}{2}-1}+1,n,\widehat{n+1},1\right\rangle .
\end{align*}
	However the last equality cannot hold. Indeed, since $Y=C\mathcal{Z}$
	and $C$ is positive, the strict coarse boundary conditions give ${\rm sign}\left\langle Y,i_{1},i_{1}+1,\dots,i_{\frac{m}{2}-1},i_{\frac{m}{2}-1}+1,n,\widehat{n+1},1\right\rangle =\left(-1\right)^{k+1}$.
	On the other hand, since $Y=C^{'}\mathcal{Z}^{'}$ it follows from
	Eq.~(\ref{eq: non vanishing boundary twistor}) and the coarse boundary
	conditions that ${\rm sign}\left\langle i_{1},i_{1}+1,\dots,i_{\frac{m}{2}-1},i_{\frac{m}{2}-1}+1,\widehat{n},n+1,1\right\rangle =\left(-1\right)^{k+1}$.
	This concludes the proof.
\end{proof}

\subsection{The winding number for $n=k+m$}

Suppose $m$ even. We show that there exists $\mathcal{Z}\in{\rm Mat}_{n,n}^{>0}$
and $C\in\grsp{k,k+m}$ such that the winding number is
\begin{equation}
	w_{n=k+m,k,m}\left(C\mathcal{Z},\mathcal{Z}\right)={\left\lfloor \frac{k+m-1}{2}\right\rfloor  \choose \frac{m}{2}}.\label{eq: explicit expression crossing n=00003Dk+m}
\end{equation}
Since we showed Proposition~\ref{prop:winding_independent_CZ} and Proposition~\ref{prop:winding_independent_n} that the winding number is independent of $\mathcal{Z}\in{\rm Mat}_{n,n}^{>0}$,
of $Y\in\ampliwb$ and of $n$, this proves Theorem~\ref{thm: winding constant}.

\paragraph{Step $1$.}

Since $C\in\grsp{k,k+m}$, it follows from the strict coarse boundary
conditions that each simplex involved in $P\left(Y,\mathcal{Z}\right)$
is of full dimension equal to $m-1$. Then we choose $0<\mu\ll1$ such
that the ray issued from the vector 
\[
	Z_{*}=Z_{n}+\mu Z_{n-1}+\cdots+\mu^{m-1}Z_{n-m+1}
\]
avoids the $\left(m-2\right)$-skeleton of $P\left(Y,\mathcal{Z}\right)$
(that is $P\left(Y,\mathcal{Z}\right)$ minus its maximal cells) in
$V_{Y}$. This is always possible since $\left(Z_{n},\dots,Z_{n-m+1}\right)$
is a basis of $V_{Y}$ and the skeleton is of codimension $2$. Thus,
the ray issued from $Z_{*}$ can only intersect a simplex of $P\left(Y,\mathcal{Z}\right)$
in its interior. Hence, the winding number is given by counting the
number of simplices intersected in their interior by the ray $Z_{*}$,
counted with a weight $\pm1$ depending on the orientation of the
simplex. More precisely, let $I$ be either a list $\left(i_{1},i_{1}+1,\dots,i_{\frac{m}{2}},i_{\frac{m}{2}}+1\right)$
of strictly ascending integers between $1$ and $n$, or a list $\left(i_{1},i_{1}+1,\dots,i_{\frac{m}{2}-1},i_{\frac{m}{2}-1}+1,n,1\right)$
of integers such that $i_{1},i_{1}+1,\dots,i_{\frac{m}{2}-1},i_{\frac{m}{2}-1}+1$
are strictly ascending between $2$ and $n-1$. We define the elementary
winding $w_{I}^{ele}$ to be $+1$ if for every $j\in\left\{ 1,m\right\} $
we have 
\begin{equation}
	{\rm sign}\left\langle Z_{*},I\backslash\left\{ j{\rm th}\,{\rm element}\right\} \right\rangle =\left(-1\right)^{j}{\rm sign}\left\langle I\right\rangle \label{eq: def elementary winding number}
\end{equation}
and $0$ otherwise. Thus, the elementary winding number $w_{I}^{ele}$
is $1$ precisely if the ray $Z_{*}$ hits the simplex $S\left(I\right)$
in its interior. Then, we have 
\[
	w_{k,m}=\sum_{I}{\rm sign}\left\langle I\right\rangle \times w_{I}^{ele},
\]
where the summation is over the lists described above. Since $C\in\grsp{k,k+m}$,
we have \\
${\rm sign}\left\langle i_{1},i_{1}+1,\dots,i_{\frac{m}{2}},i_{\frac{m}{2}}+1\right\rangle =+1$
and the negative weights in the summation only appear for $I=\left(i_{1},i_{1}+1,\dots,i_{\frac{m}{2}-1},i_{\frac{m}{2}-1}+1,n,1\right)$
and $k$ even.  

\paragraph{Step $2$ .}

In order to compute the winding number, it suffices to compute signs
of certain twistor coordinates. Since $n=k+m$, this boils down computing the sign of determinants of square $k+m$ matrices. 

If $I=\left(i_{1},i_{1}+1,\dots,i_{\frac{m}{2}},i_{\frac{m}{2}}+1\right)$,
we have
\begin{equation}
	{\rm sign}\left\langle Z_{*},I\backslash\left\{ i_{j}+\epsilon\right\} \right\rangle =\left(-1\right)^{k+i_{j}+\left(1-\epsilon\right)},\label{eq: sign for simplex of type a}
\end{equation}
$j\in\left\{ 1,\dots,\frac{m}{2}\right\} $ and $\epsilon\in\left\{ 0,1\right\} $.
Indeed, we have
\[
	\left\langle Z_{*},i_{1},i_{1}+1,\dots,\widehat{i_{j}+\epsilon},\dots,i_{\frac{m}{2}},i_{\frac{m}{2}}+1\right\rangle =\det\left(\begin{matrix}C\\
	V_{*}\\
	I_{i_{1},i_{1}+1,\ldots,\widehat{i_{j}+\epsilon},\ldots,i_{\frac{m}{2}},i_{\frac{m}{2}}+1}
	\end{matrix}\right)\det(\mathcal{Z}),
\]
where $V_{*}=\left(0,\dots,0,\mu^{n-m+1},\dots,\mu^{0}\right)$ and
$I_{i_{1},i_{1}+1,\ldots,\widehat{i_{j}+\epsilon},\ldots,i_{\frac{m}{2}},i_{\frac{m}{2}}+1}$
is the $\left(m-1\right)\times\left(k+m\right)$ matrix whose $l$th
row has a $1$ at the $l$th index of the list $i_{1},i_{1}+1,\ldots,\widehat{i_{j}+\epsilon},\ldots,i_{\frac{m}{2}},i_{\frac{m}{2}}+1$
and zeros elsewhere. We first expand the determinant along the rows of $I$; the only row possibly
contributing with a sign to the determinant is the one with a $+1$ in the position
$i_j+\left(1-\epsilon\right)$, since the sign contributions coming from the development of the other rows of $I$ compensate two by two. We obtain
\begin{align*}
	\det\left(\begin{matrix}C\\
	V_{*}\\
	I_{i_{1},i_{1}+1,\ldots,\widehat{i_{j}+\epsilon},\ldots,i_{\frac{m}{2}},i_{\frac{m}{2}}+1}
	\end{matrix}\right) & =\left(-1\right)^{k+i_{j}+\left(1-\epsilon\right)}\det\left(\begin{matrix}C_{[n]\setminus\{i_{1},i_{1}+1,\ldots,\widehat{i_{j}+\epsilon},\ldots,i_{\frac{m}{2}},i_{\frac{m}{2}}+1\}} \\
	\left(V_{*}\right)_{[n]\setminus\{i_{1},i_{1}+1,\ldots,\widehat{i_{j}+\epsilon},\ldots,i_{\frac{m}{2}},i_{\frac{m}{2}}+1\}}
	\end{matrix}\right).
\end{align*}
Finally, the sign of the determinant on the RHS is determined by the
sign of its smallest order in $\mu$ since $0<\mu\ll1$. There are
two cases depending on $M=\max\left(\left[n\right]\backslash I\right)$.
If $i_{j}+1<M$ this determinant is given by 
\[
	\mu^{n-M}\underset{=+1}{\underbrace{\left(-1\right)^{k+1+1+M}}}\det\left(C_{[n]\setminus\{i_{1},i_{1}+1,\ldots,\widehat{i_{j}+\epsilon},\dots,i_{\frac{m}{2}},i_{\frac{m}{2}}+1\}\cup\left\{ M\right\} }\right)+o\left(\mu^{n-M}\right),
\]
where we used $M\equiv n\equiv k\mod2$ to simplify the sign. If $i_{j}>M$,
then the determinant is given by 
\[
	\mu^{n-\left(i_{j}+\epsilon\right)}\underset{=+1}{\underbrace{\left(-1\right)^{k+1+i_{j}+\epsilon+\epsilon}}}\det\left(C_{[n]\setminus\{i_{1},i_{1}+1,\dots,i_{\frac{m}{2}},i_{\frac{m}{2}}+1\}}\right)+o\left(\mu^{n-\left(i_{j}+\epsilon\right)}\right),
\]
where we used that in this case $i_{j}\equiv n-1\equiv k-1\mod2$.
Thus, we deduce Eq.~(\ref{eq: sign for simplex of type a}) from
the positivity of $C$ and $\mathcal{Z}$. 

If $I=\left(i_{1},i_{1}+1,\dots,i_{\frac{m}{2}-1},i_{\frac{m}{2}-1}+1,n,1\right)$,
we obtain in a similar way
\begin{equation}
	{\rm sign}\left\langle Z_{*},I\backslash\left\{ i_{j}+\epsilon\right\} \right\rangle =\left(-1\right)^{i_{j}+\left(1-\epsilon\right)},\label{eq: sign for simplex of type b}
\end{equation}
where $j\in\left\{ 1,\frac{m}{2}-1\right\} $ and $\epsilon\in\left\{ 0,1\right\} $.
We also have 
\begin{equation}
	{\rm sign}\left\langle Z_{*},I\backslash\left\{ n\right\} \right\rangle =+1\quad{\rm and}\quad{\rm sign}\left\langle Z_{*},I\backslash\left\{ 1\right\} \right\rangle =\left(-1\right)^{k+1}\label{eq: sign for simplex of type b n1}.
\end{equation}

\paragraph{Step $3$. }

We now count the number of simplices contributing to the winding number.
There are two types of simplices: simplices of type $\left(a\right)$ associated
to lists $\left(i_{1},i_{1}+1,\dots,i_{\frac{m}{2}},i_{\frac{m}{2}}+1\right)$
of strictly ascending integers between $1$ and $n$, and simplices
of type $\left(b\right)$ associated to lists $\left(i_{1},i_{1}+1,\dots,i_{\frac{m}{2}-1},i_{\frac{m}{2}-1}+1,n,1\right)$
of integers such that $i_{1},i_{1}+1,\dots,i_{\frac{m-1}{2}},i_{\frac{m-1}{2}}+1$
are strictly ascending between $2$ and $n-1$. 

If $k$ is odd, it follows from Eq.~(\ref{eq: sign for simplex of type b n1})
that a simplex of type $\left(b\right)$ cannot satisfy Eq.~(\ref{eq: def elementary winding number}).
However, it follows from Eq.~(\ref{eq: sign for simplex of type a})
that simplex of type $\left(a\right)$ contributes by $+1$ to the
winding number if and only if 
\[
	i_{1}\equiv i_{2}\equiv\cdots\equiv i_{\frac{m}{2}}\equiv0\mod2,
\]
where $1\leq i_{1}\leq\dots\leq i_{\frac{m}{2}}\leq n-1$. Thus, the
winding number is equal to the number of such sequences $i_{1},\ldots,i_{\frac{m}{2}}$,
there are
\[
	{\frac{n-1}{2} \choose \frac{m}{2}}={\frac{k+m-1}{2} \choose \frac{m}{2}}
\]
possibilities. This ends the proof of Eq.\,(\ref{eq: explicit expression crossing n=00003Dk+m})
for $k$ odd. 

If $k$ is even, a simplex of type $\left(a\right)$ contributes by
$+1$ to the winding number if and only if
\[
	i_{1}\equiv i_{2}\equiv\cdots\equiv i_{\frac{m}{2}}\equiv1\mod2,
\]
where $1\leq i_{1}\leq\dots\leq i_{\frac{m}{2}}\leq n-1$. There are
${\frac{n}{2} \choose \frac{m}{2}}$ possibilities. Moreover, a simplex
of type $\left(b\right)$ contributes by $-1$ to the winding number
if and only if 
\[
	i_{1}\equiv i_{2}\equiv\cdots\equiv i_{\frac{m}{2}-1}\equiv0\mod2,
\]
where $2\leq i_{1}\leq\dots\leq i_{\frac{m}{2}-1}\leq n-2$. There
are ${\frac{n-2}{2} \choose \frac{m}{2}}$ possibilities. Thus, the
winding number is equal to 
\[
	{\frac{n}{2} \choose \frac{m}{2}}-{\frac{n-2}{2} \choose \frac{m}{2}-1}={\frac{n}{2}-1 \choose \frac{m}{2}},
\]
that is Eq.~(\ref{eq: explicit expression crossing n=00003Dk+m})
for $n=k+m$ and $k$ even.

This ends the proof of Theorem~\ref{thm: winding constant}.

\subsection{Maximality of the winding number for $m=2$}

In this subsection, we prove Proposition~\ref{prop: maximality} that we recall now. Fix $Y\in\grwcb{k,k+2}$ and $\mathcal{Z}\in{\rm Mat}_{n,k+2}^{>0}$. We prove in the first part that 
\[
	w_{n,k,2}\left(Y,\mathcal{Z}\right)\leq\left\lfloor \frac{k+1}{2}\right\rfloor ,
\]
and in the second part that if $w_{n,k,2}\left(Y,\mathcal{Z}\right)=\left\lfloor \frac{k+1}{2}\right\rfloor $ and satisfies the coarse boundary conditions Eq.~(\ref{eq: coarse boundary m even}), then $Y\in\mathcal{A}_{n,k,2}$. 

\paragraph{Proof of part $1$. }
First, there exists $Y^{'}\in{\rm Gr}_{k,k+2}^{{\rm wcb}}$ in the connected component of $Y$ in ${\rm Gr}_{k,k+2}^{{\rm wcb}}$ such that: 

\begin{itemize}
	\item $Z_{i}^{'}=\pi_{Y^{'}}\left(\mathcal{Z}_{i}\right)$ is nonzero in $V_{Y^{'}}$,
	\item any two vectors $\left(Z_{i}^{'},Z_{j}^{'}\right)$, for $i\neq j$, are non-collinear,
	\item $w_{n,k,2}\left(Y^{'},\mathcal{Z}\right)=w_{n,k,2}\left(Y,\mathcal{Z}\right)$. 
\end{itemize}
Indeed, each connected component of ${\rm Gr}_{k,k+2}^{{\rm wcb}}$ is of full-dimension $2k$ since the locus $\mathcal{L}$ defined in Section~\ref{sec:Coarse_boundary} is of dimension $1$. Thus, we can choose a point in the connected component of $Y$ in ${\rm Gr}_{k,k+2}^{{\rm wcb}}$ generic enough to satisfy the first two conditions. Moreover, any two points in the same connected component of ${\rm Gr}_{k,k+2}^{{\rm wcb}}$ have the same winding number (see Remark~\ref{rem:constancy-winding-Gr}). 
We now compute the winding number of $Y$. Since the winding numbers of $Y$ and of $Y^{'}$ are equal, up to redefining $Y:=Y^{'}$, we suppose that $Y$ satisfies the two first conditions. 

Denote by $s$ the number of sign flips of $\left(\left\langle Y,1,i\right\rangle \right)_{i\in\left[n\right]}$.
We show that 
\[
	\begin{cases}
		2w_{n,k,2}\left(Y,\mathcal{Z}\right)=s+1 & {\rm for\;}s\;{\rm odd},  \\
		2w_{n,k,2}\left(Y,\mathcal{Z}\right)=s   & {\rm for\;}s\;{\rm even}. 
	\end{cases}
\]
Then, it follows from Remark~\ref{rem: max sign flips} that $s\leq k$. Thus we deduce that $w_{n,k,2}\left(Y,\mathcal{Z}\right)\leq\left\lfloor \frac{k+1}{2}\right\rfloor $. 


Choose $Z_{*}$ in $V_{Y}\backslash\left\{ 0\right\} $ in a small neighborhood of $Z_{1}$ such that the line $l$ generated by $Z_{*}$ avoids all the points $Z_{i}$, for $1\leq i\leq n$. This is always possible since $\cup_{1\leq i\leq n}{\rm span}\left(Z_{i}\right)$ is of codimension $1$ in $V_{Y}$. Moreover, since no pair of vector $(Z_i,Z_j)$, for $i\neq j$, is collinear, the line $l$ can only intersect a simplex $S\left(i,i+1\right)$ (resp. $S\left(1,n\right)$) transversally and in its relative interior. In particular, the line intersects the simplex $S\left(i,i+1\right)$, for $1\leq i\leq n-1$, if and only if $\textrm{sign} \left\langle Y,Z_{*},Z_{i}\right\rangle =-\textrm{sign}\left\langle Y,Z_{*},Z_{i+1}\right\rangle $ (resp. the simplex $S\left(1,n\right)$ if and only if $\textrm{sign}\left\langle Y,Z_{*},Z_{n}\right\rangle =-\textrm{sign}\left\langle Y,Z_{*},Z_{1}\right\rangle $). We say that this intersection is positive or negative if in addition we have $\textrm{sign}\left\langle Y,i,i+1\right\rangle $ is positive or negative (resp. $\textrm{sign}\left\langle Y,1,n\right\rangle $ is positive or negative). Then, by the definition of the winding number, we have 
\[
	2w_{n,k,2}\left(Y,\mathcal{Z}\right)=\left|\sum_{x\in l\cap P\left(Y,\mathcal{Z}\right)}\epsilon\left(x\right)\right|,
\]
where $\epsilon\left(x\right)$ is $+1$ if the intersection of $l$
and $P\left(Y,\mathcal{Z}\right)$ at $x$ is positive and $-1$ if
it is negative. In particular, we have
\[
	2w_{n,k,2}\left(Y,\mathcal{Z}\right)\leq\sum_{x\in l\cap P\left(Y,\mathcal{Z}\right)}1.
\]
Moreover, the number of intersections of $l$ with $P\left(Y,\mathcal{Z}\right)$ is $s$ or $s+1$ depending on the parity of $s$. Indeed, by choosing $Z_{*}$ close enough to $Z_{1}$, we get that the number of sign flips $s$ of $\left(\left\langle Y,1,i\right\rangle \right)_{i\in\left[n\right]}$ counts the number of intersections of $l$ with the simplices $S\left(i,i+1\right)$ for $1\leq i\leq n-1$. If $s$ is even, then the list $\left(\left\langle Y,1,i\right\rangle \right)_{i\in\left[n\right]}$ has an even number of sign flips, and then ${\rm sign} \left\langle Y,Z_{1},Z_{2}\right\rangle ={\rm sign}\left\langle Y,Z_{1},Z_{n}\right\rangle $. By choosing $Z_*$ close enough to $Z_1$, we also get ${\rm sign} \left\langle Y,Z_{*},Z_{2}\right\rangle ={\rm sign}\left\langle Y,Z_{*},Z_{n}\right\rangle $. We can in addition choose $Z_{*}$ such that $Z_{1}$ and $Z_{2}$ belong to the same half plane of $\mathbb{R}^{2}\backslash l$. Then, $l$ does not intersect $S\left(1,2\right)$ and $S\left(1,n\right)$, so the number of intersections of $l$ with $P\left(Y,\mathcal{Z}\right)$ is $s$. Similarly, if $s$ is odd, then with the same choice of $Z_{*}$ we see that $l$ intersects $S\left(1,n\right)$ and then the number of intersections of $l$ with $P\left(Y,\mathcal{Z}\right)$ is $s+1$.

\paragraph{Proof of part $2$.}

Now suppose that $Y\in\grwcb{k,k+2}$, satisfies the coarse boundary
conditions and $w_{n,k,2}\left(Y,\mathcal{Z}\right)=\left\lfloor \frac{k+1}{2}\right\rfloor $.
This implies by the first part of the proof that the number of sign
flips $s$ of $\left(\left\langle Y,1,i\right\rangle \right)_{i\in\left[n\right]}$
is maximal and equal to $k$. Hence, it follows from Theorem~5.1
of \cite{parisi2021m} that $Y\in\mathcal{A}_{n,k,2}$.

\section{The crossing number\label{sec: crossing number}}
The goal of this section is to prove Theorem~\ref{thm: crossing is constant}. The structure of the proof is similar to the case of the winding number: we first prove in Subsections~\ref{subsec: C and Z equations}, \ref{subsec: Consequences of C and Z},~\ref{subsec:Properties-of-simplices} and ~\ref{subsec:Proving the crossing thm} that the crossing number is independent of the point in $\amplio \times \matp$, 
then we prove in Subsection~\ref{subsec:Independence n crossing} that the crossing number is independent of $n$ and finally we obtain in Subsection~\ref{subsec: crossing for n=00003Dk+m} the explicit expression of the crossing number for $n=k+m$. 

The first part of the proof, that is the constancy of the crossing number in $\amplio \times \matp$, is structured as follows. In Subsection~\ref{subsec: C and Z equations} we derive two types of
equations on the twistor coordinates: the $C$-equations and the $\mathcal{Z}$-equations.
These equations involve the Pl\"ucker coordinates of $C$ and $\mathcal{Z},$
and are obtained from the Pl\"ucker relations. We emphasize that they
are valid for any $n,k,m$. From the $C$- and $\mathcal{Z}$- equations,
together with the positivity of the minors of $C$ and $\mathcal{Z}$,
we deduce in Subsection~\ref{subsec: Consequences of C and Z} nontrivial
constraints on the twistor coordinates. These constraints are used
in Subsection~\ref{subsec:Properties-of-simplices} to exclude some
configurations of simplices relative to the origin in the projected
space. These are precisely the configurations where the crossing number
can change. We then conclude in Subsection~\ref{subsec:Proving the crossing thm}
that the crossing number is constant in $\amplio$.
\subsection{The $C$-equations and the $\mathcal{Z}$-equations\label{subsec: C and Z equations}}

In this section, we derive two sets of equations involving the twistor
coordinates: the $C$-equations and the $\mathcal{Z}$-equations.
These equations are obtained in the same way; they follow from Pl\"ucker
relations in $\gr{k,n}$ and in $\gr{k+m,n}$. In particular, the
positivity of the minors of $C$ and $\mathcal{Z}$ is not necessary.
Moreover, these equations are valid for any $m$, even or odd.

\subsubsection{The $C$-equations}
We recall that $\sub k$ denotes the set of lists of elements of $\left[n\right]$  of size $k$ sorted in ascending order. If $A$ and $B$ are two lists of elements of $[n]$, we denote by $A,B$ the concatenation of the two lists, and by $A\cup B$ the list obtained by sorting in ascending order the elements of $A$ and of $B$.  
We also recall that, if $I\in\sub k$ and $V\in\gr{k,n}$ is represented by a $k\times n$ matrix $M$, then the Pl\"ucker coordinates $p_{I}\left(V\right)$ denotes the minor of $M$ relative to $I$, and they do not depend on the choice of $M$ up to simultaneous rescaling by a nonzero constant.
\begin{prop}
	[$C$-equations]Let $\left(n,k,m\right)$ be a triplet of nonnegative
	integers such that $k+m\leq n$. Let $\mathcal{Z}$ be a $n\times\left(k+m\right)$
	matrix of rank $k+m$, let $C\in{\rm Gr}_{k,n}$ and
	$Y=\tilde{\mathcal{Z}}\left(C\right)$.Let $A\in\sub{k-1}$ and $B\in\sub{m-1}$.
	The $C$-equations are
	\begin{equation}
		\sum_{i=1}^{n}p_{A,i}\left(C\right)\left\langle Y,B,i\right\rangle =0,\label{The C equations}
	\end{equation}
	where we used Notation~\ref{Notation: twistor coordinates} for the
	twistor coordinate $\left\langle Y,B,i\right\rangle$.
\end{prop}
\begin{rem}
	Both the vector of Pl\"ucker coordinates $(p_{A,i}\left(C\right))_{A,i}$ and the vector of twistor coordinates
	$(\left\langle B,i\right\rangle)_{B,i} $ are defined up to a non zero
	multiplicative scalar, corresponding to the choice of the matrix representing $C$
	and $Y$. However these scalars do not affect Eq.~(\ref{The C equations}).
\end{rem}

We now prove the $C$-equations.
\begin{proof}
	From Lemma~\ref{lem: Consequence Cauchy Binet} we get
	\begin{equation}
		\sum_{i=1}^{n}p_{A,i}\left(C\right)\left\langle Y,B,i\right\rangle =\sum_{i=1}^{n}\sum_{J=\left(j_{1}<\cdots<j_{k}\right)\in\sub k}p_{A,i}\left(C\right)p_{J}\left(C\right)\left\langle J,B,i\right\rangle .\label{eq: proof C eq}
	\end{equation}
	Fix $L=\left\{ l_{1}<\cdots<l_{k+1}\right\} \in\sub{k+1}$. We now extract the coefficient of the determinant $\left\langle l_{1},\dots,l_{k+1},B\right\rangle $
	in the RHS of Eq.~(\ref{eq: proof C eq}). To do so, the
	index $i$ must be in $L$ and once it is fixed, we have
	$J=L\backslash\left\{ i\right\} $. Taking care of the sign given
	by the antisymmetry of the determinant, we write the coefficient of
	$\left\langle l_{1},\dots,l_{k+1},B\right\rangle $ in the RHS of
	Eq.~(\ref{eq: proof C eq}) as
	\[
		\left(-1\right)^{k+m}\left(\sum_{\alpha=1}^{k+1}\left(-1\right)^{\alpha}p_{A,l_{\alpha}}\left(C\right)p_{l_{1},\dots,\widehat{l_{\alpha}},\dots,l_{k+1}}\left(C\right)\right).
	\]
	This expression vanishes by the Pl\"ucker relations (see for instance \cite{griffiths2014principles}). This ends the proof. 
\end{proof}

\subsubsection{The $\mathcal{Z}$-equations}
\begin{prop}
	[$\mathcal{Z}$-equations]Let $\left(n,k,m\right)$ be a triplet of nonnegative integers such that $k+m < n$. Let $\mathcal{Z}$ be a $n\times\left(k+m\right)$ matrix of rank $k+m$ and denote by $W=\mathcal{Z}^{T}$ the $\left(k+m\right)$-plane in $\mathbb{R}^{n}$ generated by the columns of $\mathcal{Z}$. Let $C\in{\rm Gr}_{k,n}$ and $Y=\tilde{\mathcal{Z}}\left(C\right)$. Let $A\in\sub{k+m+1}$ and $B\in\sub{m-1}$. The $\mathcal{Z}$-equations are
	\begin{align}
		\sum_{i\in A}\left(-1\right)^{\#i}p_{A\backslash i}\left(W\right)\langle Y,B,i\rangle & =0,\label{The Z equation} 
	\end{align}
	where $\#i$ is the position of $i$ in the list $A$.
\end{prop}
\begin{proof}
	Using Lemma~\ref{lem: Consequence Cauchy Binet} we get
	\[
		\sum_{i\in A}\left(-1\right)^{\#i}p_{A\backslash i}\left(W\right)\langle Y,B,i\rangle=\sum_{i\in A}\sum_{J\in\sub k}\left(-1\right)^{\#i}p_{J}\left(C\right)p_{A\backslash i}\left(W\right)\left\langle J,B,i\right\rangle .
	\]
	Fix $J=\left(j_{1}<\cdots<j_{k}\right)\in\sub k$, the coefficient
	of $p_{J}\left(C\right)$ is
	\[
		\sum_{i\in A}\left(-1\right)^{\#i}p_{A\backslash i}\left(W\right)\left\langle J,B,i\right\rangle ,
	\]
	and we can write this expression as
	\[
		\sum_{i\in A}\left(-1\right)^{\#i}p_{A\backslash i}\left(W\right)p_{J,B,i}\left(W\right).
	\]
	This last expression vanishes by Pl\"ucker relations (see \cite{griffiths2014principles}) in ${\rm Gr}_{k+m,n}$.
	This proves the $\mathcal{Z}$-equations.
\end{proof}

\subsection{Consequences of the $C$- and $\mathcal{Z}$- equations on twistor
	coordinates\label{subsec: Consequences of C and Z}}

From now on, we suppose that $m=2r-1$ for $r\geq1$. The purpose
of this section is to deduce from the $C$- and $\mathcal{Z}$- equations,
together with positivity of the minors of $C$ and $\mathcal{Z},$
the following constraints on twistor coordinates.
\begin{prop}
	\label{cor: forbidden vanishings}Let $Y \in \amplio$ and
	let $B=\left(b_{1},b_{1}+1,\dots,b_{r-1},b_{r-1}+1\right)$ be a list
	of $2r-2=m-1$ integers in $\left[n\right]$.
	\begin{enumerate}
		\item The list $\left(\left\langle Y,B,i\right\rangle \right)_{i\in\left[n\right]\backslash B}$
		      cannot contain three consecutive twistor coordinates\\
		      $\left\langle Y,B,i_{0}^{-}\right\rangle ,\left\langle Y,B,i_{0}\right\rangle ,\left\langle Y,B,i_{0}^{+}\right\rangle $,
		      such that
		      \[
		      	\left\langle Y,B,i_{0}\right\rangle =0{\rm \,\,\,\,and\,\,\,\,}{\rm sign}\left(\left\langle Y,B,i_{0}^{-}\right\rangle \right)={\rm sign}\left(\left\langle Y,B,i_{0}^{+}\right\rangle \right)\neq0,
		      \]
		      where the indices $\left(i_{0}^{-},i_{0},i_{0}^{+}\right)$ satisfy:
		      $2\leq i_{0}\leq n-1$ and
		      \[
		      	i_{0}^{-}=\max\left(\left[1,\dots,i_{0}-1\right]\cap B^{c}\right),
		      \]
		      \[
		      	i_{0}^{+}=\min\left(\left[i_{0}+1,\dots,n\right]\cap B^{c}\right),
		      \]
		      where $B^{c}$ is the complement of $B$ in $\left[n\right]$.
		\item The list $\left(\left\langle Y,B,i\right\rangle \right)_{i\in\left[n\right]\backslash B}$
		      cannot contain two consecutive zeros.
	\end{enumerate}
\end{prop}
We first establish, in Section~\ref{subsec: sign flip}, a sign flip
property of the twistor coordinates. Then, in Section~\ref{subsec: proof prop forbidden vanishing},
we deduce from this property, together with the $\mathcal{Z}$-equations,
the proof of Proposition~\ref{cor: forbidden vanishings}.

\subsubsection{\label{subsec: sign flip}A sign flip property}

The following lemma was stated in \cite[Section 5]{arkani2018unwinding} and a sketch of a proof was given. It will be used in what follows. For completeness we provide another proof, which is based upon the $C$-, $\mathcal{Z}$-equations and the positivity of the minors.

\begin{lem}
	\label{lem: k sign flips}Let $Y\in\mathcal{A}_{n,k,m}^{>0}$ and
	let $B=\left(b_{1},b_{1}+1,\dots,b_{r-1},b_{r-1}+1\right)$ be a list
	of $m-1$ integers in $\left[n\right]$. Then, the number of times
	the list of numbers $\left(\left\langle Y,B,i\right\rangle \right)_{i\in\left[n\right]}$
	changes sign (ignoring the zeros) is exactly $k$.
\end{lem}

\begin{rem}
	If $m=2r$ is even, we can prove in the same way that if \\$B=\left(1,b_{1},b_{1}+1,\dots,b_{r-1},b_{r-1}+1\right)$
	or $B=\left(b_{1},b_{1}+1,\dots,b_{r-1},b_{r-1}+1,n\right)$, then
	the number of times the list $\left(\left\langle Y,B,i\right\rangle \right)_{i\in\left[n\right]}$
	changes sign is exactly $k$.
\end{rem}

We now prove the sign flip property.

\begin{proof}
	[Proof of Lemma \ref{lem: k sign flips}]The proof splits in two steps.
		
	\paragraph{Step~$1$.}
		
	In this step we use the $C$-equations and the positivity of $C,$ to deduce that
	the list $\left(\left\langle Y,B,i\right\rangle \right)_{i\in\left[n\right]},$ either
	changes sign at least $k$ times, or it is a list of zeros. The second possibility, that
	$\left(\left\langle Y,B,i\right\rangle \right)_{i\in\left[n\right]}$
	is a list of zeros, is ruled out since the first and last element of the list do not satisfy the strict inequalities of the coarse boundary conditions (see Eq.~(\ref{eq: coarse boundary m odd})). We emphasize that a more general sign flip statement has already been proven using different tools in \cite[Corollary~3.21]{karp2019amplituhedron}, more details are provided in Remark~\ref{rem: sign flips by Karp and Williams}. 
		
	Suppose $\left(\left\langle Y,B,i\right\rangle \right)_{i\in\left[n\right]}$
	changes sign $q$ times with $q<k$. Denote by $s_{1}<\dots<s_{q}$
	the positions of the sign flips. More precisely, let $s_{0}$ be the smallest
	integer such that $\left\langle Y,B,s_{0}\right\rangle \neq0$ and set
	\[
		s_{i}=\min\left(s\in\left[n\right]\mid s>s_{i-1}\ {\rm and}\ {\rm sign}\left(s\right)=-{\rm sign}\left(s_{i-1}\right)\right).
	\]
	Define $A=\left(a_{1},\dots,a_{k-1}\right)$ by
	\[
		a_{j}={\rm max}\left(j,s_{j-\left(k-q\right)+1}\right),\quad{\rm for}\quad1\leq j\leq k-1,
	\]
	with the convention that $s_{j}=0$ if $j\leq0$. More explicitly,
	let $j_{0}$ be the smallest index such that $a_{j_{0}}\neq j_{0}$,
	in this case $a_{j}=s_{j-\left(k-q\right)+1}$ for $j\geq j_0$, and the list
	$A$ is 
	\[
		A=\left(1,\dots,j_{0}-1,s_{j_{0}-\left(k-q\right)+1},\dots,s_{q}\right).
	\]
	Thus, the $C$-equation associated to $A$ and $B$ is
	\begin{equation}
		\sum_{i=a_{j_{0}}+1}^{a_{j_{0}+1}-1}p_{A,i}\left(C\right)\left\langle Y,B,i\right\rangle +\sum_{i=a_{j_{0}+1}+1}^{a_{j_{0}+2}-1}p_{A,i}\left(C\right)\left\langle Y,B,i\right\rangle +\cdots+\sum_{i=a_{k-1}+1}^{n}p_{A,i}\left(C\right)\left\langle Y,B,i\right\rangle =0, 
		\label{eq: C eq in the proof}
	\end{equation}
	where the summation can be empty if two sign flips are successive. The LHS is a sum of terms which are all nonnegative (or all nonpositive).
	Indeed, since $a_j$ corresponds to the position of a sign flip for $j\geq j_0$, the twistor coordinates
	$\left\langle Y,B,i\right\rangle $ for $i\in\left]a_{j},a_{j+1}\right[$ and $j\geq j_0$
	are all nonnegative (resp. all nonpositive). Moreover, the twistor coordinates $\left\langle Y,B,i\right\rangle $ for  $i\in\left]a_{j+1},a_{j+2}\right[$
	and for $i\in\left]a_{j-1},a_{j}\right[$ are all nonpositive (resp. all nonnegative). Since all the minors of
	$C$ are positive, the sign coming from the antisymmetry of the determinant $p_{A,i}\left(C\right)$ exactly compensates the change of sign
	of the twistor coordinates. 
	Thus all the terms of the LHS of Eq.~(\ref{eq: C eq in the proof}) vanish, and since the Pl\"ucker coordinates of $C$ are nonzero,  we obtain
	\[
		\left\langle Y,B,i\right\rangle =0,\quad {\rm for}\quad i\in\left[n\right]\backslash A.
	\]
	Then, fix $j_{0}\in\left[n\right]\backslash A$. For each $a\in A$, we 
	define the list $\tilde{A}$ from $A$ by first replacing $a$ by
	$j_{0}$ and then sorting the list in ascending order. The $C$-equation
	associated to $\tilde{A}$ and $B$ reads
	\[
		p_{\tilde{A},a}\left(C\right)\left\langle Y,B,a\right\rangle =0.
	\]
	Since $p_{\tilde{A},a}\left(C\right)\neq0$, we deduce that $\left\langle Y,B,a\right\rangle =0$
	for $a\in A$. Thus $\left(\left\langle Y,B,i\right\rangle \right)_{i\in\left[n\right]}$
	is a list of zeros.
	\paragraph{Step~$2$.}
		
	In this step we use the $\mathcal{Z}$-equations and the positivity of the minors of $\mathcal{Z},$ to show that the list $\left(\left\langle Y,B,i\right\rangle \right)_{i\in\left[n\right]}$ changes sign at most $k$ times. 
	
	The $\mathcal{Z}$-equations are valid for $n>k+m$; we first consider the case $k+m=n$. In this case, the list $\left(\left\langle Y,B,i\right\rangle \right)_{i\in\left[n\right]}$ has the same number of sign flips as $\left(\left\langle Y,B,i\right\rangle \right)_{i\in\left[n\right]\backslash B}$ since we only remove the zeros $\left\langle Y,B,i\right\rangle=0,\,i \in B$. Then the list $\left(\left\langle Y,B,i\right\rangle \right)_{i\in\left[n\right]\backslash B}$ is of length $k+1$ so its maximum number of sign flips is $k$. 
		
	Now suppose $n>k+m$ and $\left(\left\langle Y,B,i\right\rangle \right)_{i\in\left[n\right]}$ changes sign $q$ times with $q>k$. As in Step~$1$, we denote $s_{0}$ the smallest integer such that $\left\langle Y,B,s_{0}\right\rangle \neq0$, and by $s_{1}<\dots<s_{q}$ the positions of the sign flips. We introduce the list $$ A=\left(s_{0},\dots,s_{k+1}\right)\cup B.$$ Once again, the union of two lists is defined by first taking the union of their elements and then sorting them in ascending order. Since $\left\langle Y,B,i\right\rangle = 0$ when $i \in B$, an element of $B$ cannot be equal to $s_i$ for $i=0,\dots,k+1$, thus the list $A$ is of length $k+m+1$. Now the $\mathcal{Z}$-equation associated to $A$ is  
		\begin{equation}
		\sum_{i\in A}\left(-1\right)^{\#i}p_{A\backslash i}\left(W\right)\langle Y,B,i\rangle=0.\label{eq: Z eq in the proof}
		\end{equation} 
		This is a sum of numbers all nonnegative or all nonpositive. Indeed, first the matrix $\mathcal{Z}$, and
		hence $W$, has positive maximal minors thus $p_{A\backslash i}\left(W\right)$ is positive. Moreover, the terms of the sum vanish for $i \in B$ since in that case $\left\langle Y,B,i\right\rangle = 0$. Now, let $i$ go through $A \backslash B$, since the elements of $B$ come by pairs of consecutive integers it follows that the sign  $\left(-1\right)^{\#i}$ changes at each element of $A \backslash B$, however this change of sign is exactly compensated by the sign flip of $\left\langle Y,B,i\right\rangle$ at each element of $A\backslash B= \left( s_{1},\dots, s_{k+1}\right)$. We conclude that the sum of Eq~(\ref{eq: Z eq in the proof}) is a sum of terms of the same sign and thus they all vanish. Since the Pl\"ucker coordinates of $W$ are
		nonzero, we obtain
	\[
		\left\langle Y,B,a\right\rangle =0,\quad{\rm for}\quad a\in A.
	\]
	This contradicts the assumption $\left\langle Y,B,s_i \right \rangle \neq 0$ for $i=0,\dots,k+1$. This ends the proof.
\end{proof}
\begin{rem}
	\label{rem: sign flips by Karp and Williams}
	The assertion and proof of Step~$1$ is actually valid in a more general context: $B$ can be any element of $\sub{m-1}$, but in this case $\left(\left\langle Y,B,i\right\rangle \right)_{i\in\left[n\right]}$ can also be a list of zeros. We emphasize that Karp and Williams already proved a more general statement: this sign flip property is valid for any point of the interior of the amplituhedron, see \cite[Corollary~3.21]{karp2019amplituhedron}. We recall that $\amplio$ is included into the interior of the amplituhedron, see \cite[Lemma 9.4]{galashin2020parity}.  
\end{rem}
\begin{rem}
	\label{rem: max sign flips} The assertion and proof of Step~$2$ is still correct for $Y\in\gr{k,k+m}$.
	Indeed, we do not use the positivity of $C$ in the proof (or in the
	proof of the $\mathcal{Z}$-equations) and any element of $Y\in\gr{k,k+m}$
	can be written $Y=\tilde{C}\mathcal{Z}$ for $\tilde{C}\in\gr{k,n}$.
\end{rem}

\subsubsection{\label{subsec: proof prop forbidden vanishing}Proving Proposition~\ref{cor: forbidden vanishings}}

We begin with the first statement. According to Lemma~\ref{lem: k sign flips},
the list $\left(\left\langle Y,B,i\right\rangle \right)_{i\in\left[n\right]}$
has exactly $k$ sign flips. Denote by $S=\left(s_{1}<\cdots<s_{k}\right)$
the positions of the sign flips, defined as in the proof of Lemma~\ref{lem: k sign flips}.
We define the set $A$ of size $k+m+1$ to be the union of $S,B$
and $\left\{ i_{0},i_{0}^{+}\right\} ,$ where $i_{0}$ is as in the
statement of the proposition. The $\mathcal{Z}$-equation associated
to $A$ and $B$ is
\[
	\sum_{i\in A}\left(-1\right)^{\#i}p_{A\backslash i}\left(W\right)\langle Y,B,i\rangle=0.
\]
We now justify that all the terms of the sum are nonnegative or nonpositive. Since
the matrix $\mathcal{Z}$, and hence $W$, is positive, the Pl\"ucker
coordinates $p_{A\backslash i}\left(W\right)$ have a constant sign
for $i\in A$. We show that the change of sign $\left(-1\right)^{\#i}$
exactly compensates the change of sign of $\langle Y,B,i\rangle$.
First, by definition of $i_{0}^{-},i_{0}$ and $i_{0}^{+}$ the twistor coordinates $\left\langle Y,B,i\right\rangle $ for $i\in\left[i_{0}^{-},i_{0}^{+}\right]$
are all nonnegative or all nonpositive. Thus, $\left[i_{0}^{-},i_{0}^{+}\right]$
is contained between two successive sign flips, say $s_{j}$ and $s_{j+1}$.
More precisely, we have $\left[i_{0}^{-},i_{0}^{+}\right]\subset\left[s_{j},s_{j+1}\right[$
with ${\rm sign}\left(s_{j}\right)={\rm sign}\left(i_{0}^{-}\right)={\rm sign}\left(i_{0}^{+}\right)$.
Since $\left\langle Y,B,i\right\rangle =0$ for $i\in B$, we deduce
that when $i$ goes through $A$, the sign of $\left\langle Y,B,i\right\rangle $
only changes at $i=s_{1},\dots,s_{k}$. Moreover, between two successive
sign flips, there always are an even number of elements of $A$; it
can be pairs of successive elements of $B$ or the pair $\left(i_{0},i_{0}^{+}\right)$.
Hence the terms on the LHS of the $\mathcal{Z}$-equation are all nonnegative or all nonpositive. Thus, they all vanish. Moreover, by the positivity
of $W$, the Pl\"ucker coordinates of $W$ are nonzero and we deduce
that 
\[
	\left\langle Y,B,a\right\rangle =0,\,\,{\rm for\,}a\in A.
\]
	
Once again, we modify the elements of $A$ one by one; let $j\in\left[n\right]\backslash A$
and define $\tilde{A}$ from $A$ by first replacing a given element $a\in A$
by $j$ and then sorting the list in ascending order. Then, the $\mathcal{Z}$-equation
associated to $\tilde{A}$ and $B$ reads (up to a sign)
\[
	p_{A\backslash j}\left(W\right)\left\langle Y,B,j\right\rangle =0.
\]
Then, simplifying by the nonzero
Pl\"ucker coordinate, we obtain $\left\langle Y,B,j\right\rangle =0$. Thus,  \\$\left(\left\langle Y,B,i\right\rangle \right)_{i\in\left[n\right]}$
is a list of zeros, this contradicts $\left\langle Y,B,i_{0}^{-}\right\rangle \neq0$. 
		
We now prove the second statement. Denote by $i_{0}$ and $i_{0}+1$
the positions of two of the consecutive zeros and by $S=\left(s_{1}<\cdots<s_{k}\right)$
the positions of the sign flips. Similarly to the previous item, we use the $\mathcal{Z}$
equation with $A=S\cup B\cup\left\{ i_{0},i_0+1\right\} $ and deduce
that $\left(\left\langle Y,B,i\right\rangle \right)_{i\in\left[n\right]}$
is the zero list. This is impossible by the strict coarse boundary conditions given in Eq. (\ref{eq: coarse boundary m odd}).
	
\subsection{Properties of simplices containing the origin of $V_{Y}$\label{subsec:Properties-of-simplices}}
	
The purpose of this section is to use the constraints on the twistor coordinates obtained in the previous section to exclude some configurations of simplices containing the origin $\pi_{Y}\left(Y\right)$ of the quotient space $V_Y=\mathbb{R}^{k+m}/Y$ defined in Section~\ref{subsec: the crossing number}. The first two sections introduce some terminology for
cells and vertices. Then, in Section~\ref{subsec: first properties of simplices},
we obtain two fundamental lemmas describing the behavior of the simplices
containing the origin. At this stage, we give an idea of the proof
of Theorem~\ref{thm: crossing is constant} in a simplified context.
After some preparatory lemmas in Section~\ref{subsec:Parent simplices of a good cell},
we establish a refined version of these lemmas in Section~\ref{subsec:Configuration of cells adjacent}
which describes the configuration of cells around a cell containing
the origin.
	
\paragraph{Abuse of terminology.}
In the following, we will only use simplices involved in the count of the crossing number, that is simplices of type 
\[
	S\left(i_{1},i_{1}+1,\dots,i_{r},i_{r}+1\right),
\]
for $\left(i_{1},i_{1}+1,\dots,i_{r},i_{r}+1\right)$ a list
of strictly ascending integers between $1$ and $n$. The terminology simplex will only refer to these simplices. 
	
\subsubsection{Cells terminology}
	
\begin{defn}
	[Descendant and ancestor cells. Lineage] Let $\mathfrak{C}$ be a cell as defined in Section~\ref{subsec: the crossing number}, it is uniquely associated to the set $I$ of indices of its vertices. A \emph{descendant of $\mathfrak{C}$} is a cell associated to a strict subset of $I$. 
	If $S$ is a simplex, then a descendant of $S$ is a descendant of the $m$-cell of $S$.
			
	A cell $\mathfrak{C}^{'}$ is an \emph{ancestor} of $\mathfrak{C}$ if $\mathfrak{C}$ is a descendant of $\mathfrak{C}^{'}$. A simplex $S$ is an \emph{ancestor} of $\mathfrak{C}$ if $\mathfrak{C}$ is a descendant of $S$.
			
	The \emph{lineage of a cell $\mathfrak{C}$} is the set containing the descendent cells of $\mathfrak{C}$, the ancestor cells of $\mathfrak{C}$ and $\mathfrak{C}$. 
\end{defn}
\begin{notation}
	Let $\mathfrak{C}$ be a cell. Denote by $S_{\mathfrak{C}}$ the set of simplices which are ancestors of $\mathfrak{C}$. Denote by $L_{\mathfrak{C}}$ the set of cells of the simplices in $S_{\mathfrak{C}}$. Denote by $U_{\mathfrak{C}}\subset V_Y$ the set of points of the cells in $L_{\mathfrak{C}}$. 
\end{notation}
\begin{defn}
	[Boundary cells.  Internal cells] A \emph{boundary cell} is either
	\begin{itemize}
		\item the cell $\mathfrak{C}\left(1,i_{1},i_{1}+1,\dots,i_{r-1},i_{r-1}+1\right)$,
		      where $\left(1,i_{1},i_{1}+1,\dots,i_{r-1},i_{r-1}+1\right)$ is a
		      list of strictly ascending integers smaller or equal to $n$, or
		\item the cell $\mathfrak{C}\left(i_{1},i_{1}+1,\dots,i_{r-1},i_{r-1}+1,n\right)$,
		      where $\left(i_{1},i_{1}+1,\dots,i_{r-1},i_{r-1}+1,n\right)$ is a
		      list of strictly ascending integers greater or equal to $1$, or
		\item a descendant of one of these cells.
	\end{itemize}
	An \emph{internal cell} is a cell which is not a boundary cell.
\end{defn}
It follows from the coarse boundary conditions, Eq.~(\ref{eq: coarse boundary m odd}),
that if $Y\in\amplio$, then the origin $\pi_{Y}\left(Y\right)$ cannot
belong to a boundary cell. Hence the crossing number counts the number
of internal cells containing the origin of $V_{Y}$.
\begin{example}
	When $m=3$, the vertex $Z_{i}$ for $i=1,\dots,n$ is a boundary cell, the cell $\mathfrak{C}\left(i,i+1\right)$ for $i=1,\dots,n-1$
	is a boundary cell and the cell $\mathfrak{C}\left(i,j\right)$, with $j>i+1$,
	is an internal cell whenever $j<n-1$ and $i>2$.
\end{example}
	
\subsubsection{Conjugate vertex}
\begin{lem}
	\label{lem: conjugate vertex} Let $I \in \sub m$  and $i\in\left[n\right]$ be such that $S(I,i)$ is a simplex, and such that its descendent $(m-1)$-cell $\mathfrak{C}\left(I\right)$ is nonempty and is internal. Then, there exists a unique index $\bar{i}\neq i$ such that $S\left(I,\bar{i}\right)$ is an ancestor simplex of $\mathfrak{C}\left(I\right)$.
	
	
	
\end{lem}
\begin{notation}
	When $\mathcal{Z}$ and $Y$ are understood, the vertex $i$ refers to the vertex $Z_{i}$. 
\end{notation}
\begin{defn}
	[Conjugate vertex]The vertex $\bar{i}$ is called the \emph{conjugate
		vertex of $i$ relative to the simplex $S\left(I,i\right)$}.
\end{defn}
\begin{example}
	In Figure~\ref{fig:good and bad configurations}, the conjugate vertex
	of $i$ relative to $S\left(i,i+1,j,j+1\right)$ is $\bar{i}=i+2$.
\end{example}
\begin{proof}
	We first relabel the simplices of $S\left(I,i\right)$ by pairs of
	consecutive indices: let $j_{1},\dots,j_{r}$ in $\left[n\right]$
	such that $\left\{ j_{1},j_{1}+1,\dots,j_{r},j_{r}+1\right\} =I\cup\left\{ i\right\} $.
	\begin{itemize}
		\item If $i=j_{a}$ for $1\leq a\leq r$. Then
		      \[
		      	\bar{i}:=\min\left\{ I^{c}\cap\left\{ j_{a}+2,\dots,n\right\} \right\} ,
		      \]
		      where $I^{c}$ is the complement of $I$ in $\left[n\right]$, is
		      the conjugate vertex to $i$ relative to $S\left(I,i\right)$. Indeed,
		      first $\bar{i}$ exists since $\mathfrak{C}\left(I\right)$ is an internal cell. Then,
		      $S\left(I,\bar{i}\right)$ is a simplex: for any $1\leq a\leq r$,
		      introduce
		      \begin{align*}
		      	k_{a} & :=j_{a}+1{\rm \,\,\,\,\,if\,\,}j_{a}\in\left[i,\bar{i}\right], \\
		      	k_{a} & :=j_{a}{\rm \,\,\,\,\,\,\,\,\,\,\,\,\,\,\,\,otherwise},        
		      \end{align*}
		      then $\left\{ k_{1},k_{1}+1,\dots,k_{r},k_{r}+1\right\} $  are the
		      indices of vertices of $S\left(I,\bar{i}\right)$. Finally, no other
		      simplex can be an ancestor of $\mathfrak{C}\left(I\right)$ since it is a $\left(m-1\right)$-cell,
		      hence $\bar{i}$ is unique.
		\item If $i=j_{a}+1$ for $1\leq a\leq r$. Then
		      \[
		      	\bar{i}:=\max\left(\left\{ 1,\dots,i_{a}-1\right\} \cap I^{c}\right)
		      \]
		      is, for similar reasons, the conjugate vertex to $i$ relative to
		      $S\left(I,i\right)$.
	\end{itemize}
\end{proof}

\subsubsection{First properties of simplices containing the origin of $V_{Y}$\label{subsec: first properties of simplices}}
	
The purpose of the section is to deduce from the constraints on the
twistor coordinates obtained in Section~\ref{subsec: Consequences of C and Z}
two properties of the simplices containing the origin. From these
two properties we can already understand why the crossing number is constant
on $\amplio$ as explained at the end of the section.
\begin{lem}
	[Full-dimensional simplex]\label{lem: full-dimensional}Let $Y\in\amplio$.
	Each simplex with a cell containing $\pi_{Y}\left(Y\right)$ is full
	dimensional.
\end{lem}
In particular, for an element $Y\in\amplio$ the crossing number corresponds
to the number of simplices containing the origin $\pi_{Y}\left(Y\right)$
with the convention that if $\pi_{Y}\left(Y\right)$ is in a cell
belonging to several simplices, it is counted once.
\begin{proof}
	Suppose that $S\left(i_{1},i_{1}+1,\dots,i_{r},i_{r}+1\right)$ is a flat simplex containing the origin $\pi_{Y}\left(Y\right)$. Then, by Remark~\ref{rem: det in projected space and twistors}, for any
	list $L$ of $m$ elements of $\left\{ i_{1},i_{1}+1,\dots,i_{r},i_{r}+1\right\} $,
	we have $\left\langle Y,L\right\rangle =0$. In particular, let $B=\left(i_{1},i_{1}+1,\dots,i_{r-1},i_{r-1}+1\right)$,
	then
	\[
		\left\langle Y,B,i_{r}\right\rangle =0=\left\langle Y,B,i_{r}+1\right\rangle .
	\]
	This is forbidden by the second assertion of Proposition~\ref{cor: forbidden vanishings}.
\end{proof}
Let $\mathfrak{C}$ be a $\left(m-1\right)$ internal cell. Denote by $I$ the list of indices of vertices of $\mathfrak{C}$. By definition of a $\left( m-1 \right)$ cell, it is the descendant of a least one simplex, say $S\left(I,i\right)$, for $i \in [n] \backslash I$.  Then by Lemma~\ref{lem: conjugate vertex}, $\mathfrak{C}$ is a descendant cell of exactly two simplices $S\left(I,i\right)$ and $S\left(I,\bar{i}\right)$. Furthermore,
the hyperplane $H_{\mathfrak{C}}$ containing $\mathfrak{C}$ divides $V_{Y}$ into two
open half-spaces.
\begin{lem}
	[Main lemma]\label{lem: main lemma}Let $Y\in\amplio$. If $\pi_{Y}\left(Y\right)$
	is contained in $\mathfrak{C}$ or a descendant of $\mathfrak{C}$, then $Z_{i}$ and $Z_{\bar{i}}$
	must belong to different open half-spaces relative to $H_{\mathfrak{C}}$.
\end{lem}
An example of configurations of simplices for $m=3$ is given in Figure~\ref{fig:good and bad configurations}.
In this example $I=\left\{ i+1,j,j+2\right\} $ and $\bar{i}=i+2$.
If the origin belongs to $\mathfrak{C}\left(i+1,j,j+1\right)$ or a descendant of
this cell (in red on the figures), then the simplices can only be
in the configuration of Figure~\ref{fig:sub2}.
	
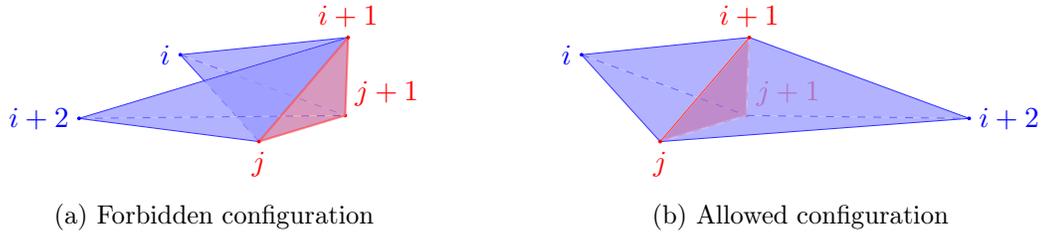
\begin{figure}[h]
	\centering
	\begin{subfigure}{.5\textwidth}
		\centering
		\begin{tikzpicture}
			\coordinate (j+1) at (-0.5,-0.5,-1.2);
			\coordinate (j) at (-0.8,0,1);
			\coordinate (i+1) at (0,1,0);
			\coordinate (k) at (-1.915,0.425,0);
			\filldraw[black] (k) circle (0.2pt);
			\coordinate (i) at (-3,0,-2);
			\coordinate (i+2) at (-3.5,0,0.2);
			\filldraw[red] (j+1) circle (0.5pt) node[anchor=south west]{$j+1$};          \filldraw[color=blue!0, fill=blue!40, rounded corners=0.5pt,opacity=0.5] (i) -- (j+1) -- (i+1) -- cycle;
			\filldraw[color=blue!0, fill=blue!40, rounded corners=0.5pt,opacity=0.5] (i) -- (j+1) -- (j) -- cycle;
			\draw [dashed, color=blue!80] (i) -- (j+1);
			\filldraw[color=blue!0, fill=blue!40, rounded corners=0.5pt,opacity=0.5] (i) -- (j) -- (i+1) -- cycle;
			\draw [color=blue!80] (i) -- (i+1);
			\draw [dashed, color=blue!80] (j) -- (k);
			\draw [color=blue!80] (k) -- (i);
			\filldraw[color=blue!0, fill=blue!40, rounded corners=0.5pt,opacity=0.5] (i+2) -- (j+1) -- (i+1) -- cycle;
			\filldraw[color=blue!0, fill=blue!40, rounded corners=0.5pt,opacity=0.5] (i+2) -- (j+1) -- (j) -- cycle;
			\draw [dashed, color=blue!80] (i+2) -- (j+1);
			\filldraw[color=blue!0, fill=blue!40, rounded corners=0.5pt,opacity=0.5] (i+2) -- (j) -- (i+1) -- cycle;
			\draw [color=blue!80] (i+2) -- (i+1);
			\draw [color=blue!80] (i+2) -- (j);
			\filldraw[color=red!90, fill=red!50, thick, rounded corners=0.5pt, opacity=0.5] (j+1) -- (j) -- (i+1) -- cycle;
			\filldraw[red] (i+1) circle (0.5pt) node[anchor=south]{$i+1$};
			\filldraw[red] (j) circle (0.5pt) node[anchor=north]{$j$};
			\filldraw[blue] (i) circle (0.5pt) node[anchor=east]{$i$};
			\filldraw[blue] (i+2) circle (0.5pt) node[anchor=east]{$i+2$};
		\end{tikzpicture}
		\caption{Forbidden configuration}
		\label{fig:sub1}
	\end{subfigure}%
	\begin{subfigure}{.5\textwidth}
		\centering
		\begin{tikzpicture}
			\coordinate (j+1) at (-0.5,-0.5,-1.2);     \coordinate (j) at (-0.8,0,1);     \coordinate (i+1) at (0,1,0);          \coordinate (i) at (-3,0,-2);     \coordinate (i+2) at (3,0,0.2);          \filldraw[red] (j+1) circle (0.5pt) node[anchor=south west]{$j+1$};          \filldraw[color=blue!0, fill=blue!40, rounded corners=0.5pt,opacity=0.5] (i) -- (j+1) -- (i+1) -- cycle;     \filldraw[color=blue!0, fill=blue!40, rounded corners=0.5pt,opacity=0.5] (i) -- (j+1) -- (j) -- cycle;     \filldraw[color=blue!0, fill=blue!40, rounded corners=0.5pt,opacity=0.5] (i+2) -- (j+1) -- (i+1) -- cycle;     \filldraw[color=blue!0, fill=blue!40, rounded corners=0.5pt,opacity=0.5] (i+2) -- (j+1) -- (j) -- cycle;          \draw [dashed, color=blue!80] (i) -- (j+1);     \draw [dashed, color=blue!80] (j+1) -- (i+2);     \draw [dashed, color=red!90] (j+1) -- (i+1);     \draw [dashed, color=red!90] (j+1) -- (j);          \filldraw[color=red!0, fill=red!80, thick, rounded corners=0.5pt, opacity=0.5] (j+1) -- (j) -- (i+1) -- cycle;          \filldraw[color=blue!0, fill=blue!40, rounded corners=0.5pt,opacity=0.5] (i) -- (j) -- (i+1) -- cycle;     \filldraw[color=blue!0, fill=blue!40, rounded corners=0.5pt,opacity=0.5] (i+2) -- (j) -- (i+1) -- cycle;          \draw [color=red!80] (i+1) -- (j);     \draw [color=blue!80] (i) -- (i+1);     \draw [color=blue!80] (i) -- (j);     \draw [color=blue!80] (i+1) -- (i+2);     \draw [color=blue!80] (j) -- (i+2);          \filldraw[red] (i+1) circle (0.5pt) node[anchor=south]{$i+1$};     \filldraw[red] (j) circle (0.5pt) node[anchor=north]{$j$};          \filldraw[blue] (i) circle (0.5pt) node[anchor=east]{$i$};     \filldraw[blue] (i+2) circle (0.5pt) node[anchor=west]{$i+2$};
		\end{tikzpicture}
		\caption{Allowed configuration}
		\label{fig:sub2}
	\end{subfigure}
	\caption{Two configurations of the simplices $S\left(i,i+1,j,j+1\right)$ and $S\left(i+1,i+2,j,j+1\right)$ ancestors of the cell $\mathfrak{C}\left(i+1,j,j+1\right)$ in $m=3$. In Figure \ref{fig:sub1}, the vertices $i$ and $i+2$ belong to the same side of the plane generated by the cell $\mathfrak{C}\left(i+1,j,j+1\right)$; in Figure \ref{fig:sub2} these vertices belong to different sides.}
	\label{fig:good and bad configurations}
\end{figure}
	
Suppose the origin $\pi_{Y}\left(Y\right)$ belongs to the cell $\mathfrak{C}\left(i+1,j,j+1\right)$.
When $Y$ moves continuously in $\amplio$, then the points $Z_{i}=\pi_{Y}\left(\mathcal{Z}_{i}\right)$
for $1\leq i\leq n$ move continuously. If the configuration of Figure~\ref{fig:sub1}
was allowed, then after a small modification of $Y$, the crossing
number could jump by $+1$ if the origin jumps into the two cells
$\mathfrak{C}\left(i,i+1,j,j+1\right)$ and $\mathfrak{C}\left(i+1,i+2,j,j+1\right)$, or
by $-1$ if the origin jumps out of any cell of the two simplices.
This bad behavior of the crossing number does not happen in the configuration
of Figure~\ref{fig:sub2}. More generally, we prove in this way that
if $\pi_{Y}\left(Y\right)$ avoids cells of codimension $2$ or more,
then the crossing number is constant in $\amplio$. The rest of the
Section~\ref{subsec:Properties-of-simplices} is devoted to making
this argument correct if $\pi_{Y}\left(Y\right)$ belongs to any internal
cell. To do so, we will only use the main lemma and the lemma about
full-dimensional simplices.
\begin{proof}
	First $Z_{i}$ (resp. $Z_{\bar{i}}$) cannot belong to $H_{\mathfrak{C}}$, otherwise the ancestor simplex $S\left(I,i\right)$ (resp. $S\left(I,\bar{i}\right)$) of $\mathfrak{C}$ is not full-dimensional, which is forbidden by Lemma~\ref{lem: full-dimensional}.
			
	Now, suppose that $Z_{i}$ and $Z_{\bar{i}}$ belong to the same open half-space with respect to $H_{\mathfrak{C}}$. Say $i<\bar{i}$, then it follows from the construction of the conjugate vertex in Lemma~\ref{lem: conjugate vertex} that $\left[i+1,\bar{i}-1\right] \subseteq I$ is nonempty. Pick $i_{0}$ in this sequence and define $B=I\backslash\left\{ i_{0}\right\} $.
	Since $\pi_{Y}\left(Y\right)$ belongs to $H_{\mathfrak{C}}$, we have
	\[
		\left\langle Y,B,i_{0}\right\rangle =0.
	\]
	Moreover, since $Z_{i}$ and $Z_{\bar{i}}$ belong to the same open half-space with respect to $H_{\mathfrak{C}}$, we have
	\[
		{\rm sign}\left(\left\langle Y,B,i\right\rangle \right)={\rm sign}\left(\left\langle Y,B,\bar{i}\right\rangle \right)\neq0.
	\]
	But this situation is impossible according to Proposition~\ref{cor: forbidden vanishings}
	(with $i_{0}^{-}=i$ and $i_{0}^{+}=\bar{i}$). The situation is similar
	for $\bar{i}<i$. This proves the lemma.
\end{proof}
	
\subsubsection{Ancestor simplices of an internal cell\label{subsec:Parent simplices of a good cell}}
	
The goal of this section is to prove the following proposition which describes the ancestor simplices of an internal cell in terms of one simplex
and its conjugate vertices. This proposition will only be used for the proof of Proposition~\ref{prop: Config adjacent cells} in the next section.
\begin{prop}
	\label{prop: parent simplices}Let $\mathfrak{C}$ be an internal cell of dimension $d<m$, and let $I=\left(i_{0},\dots,i_{d}\right)$ be the set of indices of its vertices. Let $j_{1},j_{2},\dots,j_{m-d}$ in $\left[n\right]$ be such that $S\left(I,j_{1},j_{2},\dots,j_{m-d}\right)$ is an ancestor simplex of $\mathfrak{C}$. Introduce, for each $1\leq a\leq m-d$, the conjugate vertex $\bar{j}_{a}$ of $j_{a}$ relative to $S\left(I,j_{1},j_{2},\dots,j_{m-d}\right)$. Then, for each choice of $\left(\alpha_{1},\dots,\alpha_{m-d}\right)$ in $\left\{j_{1},\bar{j}_{1}\right\}\times\cdots\times\left\{j_{m-d},\bar{j}_{m-d}\right\}$ such that the elements of $\left(\alpha_{1},\dots,\alpha_{m-d}\right)$ are all pairwise distinct, the simplex
	\[
		S\left(I,\alpha_{1},\dots,\alpha_{m-d}\right)
	\]
	is an ancestor simplex of $\mathfrak{C}$, and moreover each ancestor simplex of $\mathfrak{C}$ is obtained in this way.
\end{prop}
To prove this proposition, we first establish two lemmas which
can be safely forgotten once the proposition is proved. We use the
notations of Proposition~\ref{prop: parent simplices} for these
two lemmas.
	
We first introduce the following terminology. Let $S$ be a simplex. Its vertices are partitioned into $r$ couples with
consecutive indices. Each couple is called a \emph{pair of
	vertices of} $S$ and the corresponding couple of indices is called a \emph{pair of indices of} $S$. The following lemma explains how indices of an internal cell $\mathfrak{C}$ can be paired in an ancestor of $\mathfrak{C}$.
\begin{lem}
	\label{lem: odd and even sequences of indices}Let $\mathfrak{C}$ be an internal
	cell. Let $L=\left(i_{a},\dots,i_{b}\right)$, with $0\leq a\leq b\leq d$,
	be a sequence of consecutive isolated indices of $\mathfrak{C}$ (i.e. the vertex
	$i_{a}-1$ and the vertex $i_{b}+1$ are not vertices of $\mathfrak{C}$). Then,
	\begin{enumerate}
		\item if the length of $L$ is odd, the vertices of $L$ can appear paired in two different ways in a simplex of $S_{\mathfrak{C}}$:
		      \begin{enumerate}
		      	\item the vertex $i_{a}$ is paired with the vertex $i_{a}-1$ and the rest
		      	      of the vertices of $L$ are paired together,
		      	\item the vertex $i_{b}$ is paired with $i_{b}+1$ and the rest of the
		      	      vertices are paired together,
		      \end{enumerate}
		\item if the length of this sequence is even, then in any simplex of $S_{\mathfrak{C}}$
		      the vertices of $\left(i_{a},\dots,i_{b}\right)$ are paired together,
		      i.e. $i_{a}$ with $i_{a}+1$, $i_{a+2}$ with $i_{a+2}+1$ ...
	\end{enumerate}
\end{lem}
\begin{proof}
	The first assertion follows from the definition of a simplex. We prove
	the second assertion by contradiction: suppose $S$ is an ancestor simplex
	of $\mathfrak{C}$ and suppose that the pairings of the indices of $S$ containing
	an element of $L$ are
	\[
		P_{1}=\left(i_{a}-1,i_{a}\right),\,P_{2}=\left(i_{a}+1,i_{a}+2\right),\dots,\,P_{\frac{b-a+3}{2}}=\left(i_{b},i_{b}+1\right).
	\]
	Then, let $\mathfrak{C}^{'}$ be the $\left(m-2\right)$-cell with the same pairings
	as $S$ except $P_{1},\dots,P_{\frac{b-a+1}{2}}$ which are shifted by $+1$
	(i.e replace $P_{1}$ by $\left(i_{a},i_{a}+1\right)$ and so on),
	and $P_{\frac{b-a+3}{2}}$ which is forgotten. But $\mathfrak{C}^{'}$ is a boundary cell,
	indeed let $J$ be the indices of the vertices of $\mathfrak{C}^{'}$ and $J^{c}$
	its complement in $\left[n\right]$, so one can add to $\mathfrak{C}^{'}$ the vertex
	$\max\left(J^{c}\right)$ to obtain a cell of type $\mathfrak{C}\left(k_{1},k_{1}+1,\dots,k_{r},k_{r}+1,n\right)$.
	Moreover $\mathfrak{C}^{'}$ is an ancestor cell of $\mathfrak{C}$, hence $\mathfrak{C}$ is a boundary
	cell, which is a contradiction.
\end{proof}
The second lemma is the following.
\begin{lem}
	We have that:
	\begin{itemize}
		\item the number of vertices of $\mathfrak{C}$ satisfies $d+1\geq r$ ,
		\item the number $n_{\mathfrak{C}}$ of sequences of isolated indices of odd length
		      in $I$ is exactly $2r-\left(d+1\right)=m-d$.
	\end{itemize}
\end{lem}
\begin{proof}
	First, since $\mathfrak{C}$ is an internal cell, we have $d+1\geq r$. Otherwise,
	we can always complete the vertices of $\mathfrak{C}$ to form a boundary cell
	of dimension $\left(m-1\right)$.
			
	We then prove by induction on the number of vertices $d+1$ of $\mathfrak{C}$
	that $n_{\mathfrak{C}}=2r-\left(d+1\right)$.
			
	\paragraph{Initialization $d+1=r$.}
			
	If $n_{\mathfrak{C}}>2r-\left(d+1\right)=r$, then $\mathfrak{C}$ cannot be a descendant of a simplex because a simplex contains at most $r$ couples of isolated indices. If $n_{\mathfrak{C}}<r$, then two indices of $I$ are consecutive and $\mathfrak{C}$ has $r$ vertices but in this case it is easy to verify that $\mathfrak{C}$ cannot be an internal cell.
			
	\paragraph{Heredity $d+1>r$.}
			
	Suppose that $n_{\mathfrak{C}}$ is greater (resp. lower) than $2r-\left(d+1\right)$. Then pick one odd sequence of consecutive indices and select the maximal index of this sequence. Remove the vertex of $\mathfrak{C}$ corresponding to this index and denote by $\mathfrak{C}^{'}$ the corresponding cell. Since $\mathfrak{C}$ is an internal cell with $\left(d+1\right)$ vertices, then $\mathfrak{C}^{'}$ is an internal cell with $d$ vertices. Moreover $n_{\mathfrak{C}^{'}}=n_{\mathfrak{C}}-1$ is greater (resp. lower) than $2r-s-1$, but this is impossible by induction.
\end{proof}
We now prove the proposition.
\begin{proof}
	[Proof of Proposition \ref{prop: parent simplices}]We deduce from these two lemmas that since $S\left(I,j_{1},j_{2},\dots,j_{m-d}\right)$ is an ancestor simplex of $\mathfrak{C}$, then each sequence of isolated indices of odd length of $\mathfrak{C}$ is followed or preceded by $j_{a}$ for $1\leq a\leq m-d$. Moreover $\bar{j}_{a}$ is, by construction, the predecessor or successor
	of this same sequence. It is then a consequence of Lemma~\ref{lem: odd and even sequences of indices}
	that a simplex is an ancestor of $\mathfrak{C}$ if and only if it is associated to
	the vertices $\left(I,\alpha_{1},\dots,\alpha_{m-d}\right)$ for a
	choice of $\alpha_{1},\dots,\alpha_{m-d}$ in $\left(j_{1},\bar{j}_{1}\right)\times\cdots\times\left(j_{m-d},\bar{j}_{m-d}\right)$
	pairwise disjoint.
\end{proof}
	
\subsubsection{Configuration of ancestor simplices of a cell containing the origin\label{subsec:Configuration of cells adjacent}}
	
The following proposition explains how ancestor simplices of a cell $\mathfrak{C}$ containing the origin of $V_{Y}$ ``triangulate'' a neighborhood of $\mathfrak{C}$, i.e. cells of these simplices are pairwise distinct and they cover a neighborhood of $\mathfrak{C}$. We recall that $S_{\mathfrak{C}}$ is the set of ancestor simplices of $\mathfrak{C}$, $L_{\mathfrak{C}}$ is the set of cells of the simplices of $S_{\mathfrak{C}}$, and $U_{\mathfrak{C}}\subset V_Y$ is the set of points of the cells in $L_{\mathfrak{C}}$. 
We also recall that since the origin of $V_{Y}$ belongs to $\mathfrak{C}$, then each simplex of $S_{\mathfrak{C}}$ is full-dimensional (Lemma~\ref{lem: full-dimensional}). 
\begin{prop}
	[Configuration of adjacent cells]\label{prop: Config adjacent cells}Let
	$Y\in\amplio$. Let $\mathfrak{C}$ be a cell containing the origin $\pi_{Y}\left(Y\right)$ of $V_Y$.
	Then,
	\begin{enumerate}
		\item the set $U_{\mathfrak{C}}$ is a neighborhood of $\mathfrak{C}$ in $V_Y$, that is $U_{\mathfrak{C}}$ contains an open set of $V_Y$ containing the cell $\mathfrak{C}$,
		\item each point of $U_{\mathfrak{C}}$ belongs to a unique cell of $L_{\mathfrak{C}}$.
	\end{enumerate}
\end{prop}
\begin{example}
	Suppose $\mathfrak{C}= \mathfrak{C}\left(i,j\right)$ with $j>i+2$, $j<n$, $i>1$ and $m=3$. Then $\mathfrak{C}$ is a descendant of $4$ simplices. Following Proposition~\ref{prop: Config adjacent cells}, these $4$ simplices form a neighborhood of $\mathfrak{C}$ and their cells cannot intersect. These $4$ simplices must be in the configuration of Figure~\ref{fig: config simplices}.
\end{example}
\begin{figure}[h]
	\centering
	\begin{tikzpicture}
		\coordinate (i) at (0,0,0);     \coordinate (j) at (0,1,0);          \coordinate (i-1) at (-2,0,0.5);     \coordinate (j+1) at (0,0.2,-2);          \coordinate (i+1) at (3,0.3,0.5);          \coordinate (j-1) at (-0.7,-0.2,2);               \filldraw[color=blue!0, fill=blue!40, rounded corners=0.5pt,opacity=0.5] (i-1) -- (j) -- (j+1) -- cycle;     \draw [dashed, color=blue!80] (i-1) -- (j+1);     \draw [color=blue!80] (i-1) -- (j);     \draw [color=blue!80] (j) -- (j+1);     \filldraw[color=blue!0, fill=blue!40, rounded corners=0.5pt,opacity=0.5] (i-1) -- (j+1) -- (i) -- cycle;     \draw [dashed, color=blue!80] (j+1) -- (i);     \filldraw[color=blue!0, fill=blue!40, rounded corners=0.5pt,opacity=0.2] (i) -- (j+1) -- (j) -- cycle;     \filldraw[color=blue!0, fill=blue!40, rounded corners=0.5pt,opacity=0.2] (i) -- (j) -- (i-1) -- cycle;               \filldraw[blue] (i-1) circle (0.5pt) node[anchor=east]{$i-1$};     \filldraw[blue] (j+1) circle (0.5pt) node[anchor=south]{$j+1$};          \filldraw[color=red] (i) -- (j);     \filldraw[red] (i) circle (0.5pt) node[anchor=north]{$i$};     \filldraw[red] (j) circle (0.5pt) node[anchor=south]{$j$};          \filldraw[color=blue!0, fill=blue!40, rounded corners=0.5pt,opacity=0.5] (i) -- (j+1) -- (i+1) -- cycle;     \filldraw[color=blue!0, fill=blue!40, rounded corners=0.5pt,opacity=0.5] (j) -- (j+1) -- (i+1) -- cycle;     \filldraw[color=blue!0, fill=blue!40, rounded corners=0.5pt,opacity=0.2] (i) -- (i+1) -- (j) -- cycle;     \draw [color=blue!80] (j+1) -- (i+1);     \draw [color=blue!80] (j) -- (j+1);     \draw [dashed, color=blue!80] (i) -- (i+1);               \filldraw[color=blue!0, fill=blue!40, rounded corners=0.5pt,opacity=0.5] (i) -- (i+1) -- (j-1) -- cycle;     \filldraw[blue] (i+1) circle (0.5pt) node[anchor=west]{$i+1$};     \draw [dashed, color=blue!80] (i) -- (j-1);     \draw [dashed, color=blue!80] (i-1) -- (i);               \filldraw[color=blue!0, fill=blue!40, rounded corners=0.5pt,opacity=0.2] (i) -- (j) -- (j-1) -- cycle;     \filldraw[color=blue!0, fill=blue!40, rounded corners=0.5pt,opacity=0.5] (i) -- (i-1) -- (j-1) -- cycle;     \filldraw[color=blue!0, fill=blue!40, rounded corners=0.5pt,opacity=0.5] (j) -- (i+1) -- (j-1) -- cycle;     \filldraw[color=blue!0, fill=blue!40, rounded corners=0.5pt,opacity=0.5] (j) -- (i-1) -- (j-1) -- cycle;          \filldraw[blue] (j-1) circle (0.5pt) node[anchor=north]{$j-1$};          \draw [color=blue!80] (j) -- (i+1);     \draw [color=blue!80] (j-1) -- (i+1);     \draw [color=blue!80] (j-1) -- (j);     \draw [color=blue!80] (i-1) -- (j-1);     \draw [color=blue!80] (i-1) -- (j);
	\end{tikzpicture}
			
	\caption{Configuration of the $4$ ancestor simplices to a cell $\mathfrak{C}\left(i,j\right)$
		containing the origin when $j>i+2$, $j<n$, $i>1$ and $m=3$ .  \label{fig: config simplices}}
			
\end{figure}
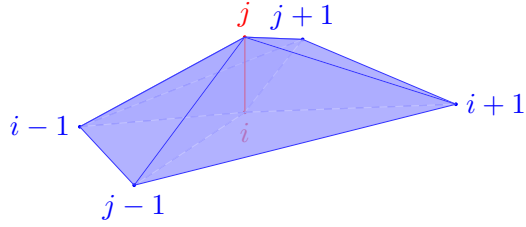
	
In particular, it follows from this proposition that for a small and continuous modification of $Y$, the number of cells of $L_{\mathfrak{C}}$ containing $\pi_{Y}\left(Y\right)$ is one.
	
The rest of this section is devoted to the proof of this Proposition.
	
\paragraph{The case $\dim\left(\mathfrak{C}\right)=m$. }
	
If $\dim\left(\mathfrak{C}\right)=m$, then the ancestor simplex $S$ of $\mathfrak{C}$ is full-dimensional and the cells of $L_{\mathfrak{C}}$ are the cells of $S$. In particular, they form a neighborhood of $\mathfrak{C}$ and they cannot pairwise intersect, otherwise $S$ is flattened and this is forbidden by Lemma~\ref{lem: full-dimensional}.

\paragraph{The case $\dim\left(\mathfrak{C}\right)<m$.}
In order to prove assertion $1$ and $2$ when $\dim\left(\mathfrak{C}\right)<m$, we need to establish some preparatory lemmas.

We define the cone $\mathcal{C}_S$ associated to a $S$ simplex of $S_\mathfrak{C}$ to be the convex cone generated by the vertices of $S$ in $V_Y$, that is the set of linear combinations of vertices of $S$ with nonnegative coefficients. Since $S$ is a convex simplex containing the origin, its relation with the cone $\mathcal{C}_S$ is particularly simple as explained by the following lemma.  

\begin{lem}
	\label{lem:ray_origin}
	Let $r$ be a ray of the convex cone $\mathcal{C}_S$ associated to a simplex $S$ of $S_\mathfrak{C}$. Then there exists a point $x$ of the ray $r$, different from the origin, such that the segment $[0,x]$ is included in $S$.

\end{lem}
\begin{proof}
	Let $r$ be a ray of $\mathcal{C}_S$. Since $\mathcal{C}_S$ is generated by $S$, then $S$ contains at least one point $x$ of the ray $r$ which is different from the origin. Now, since $S$ is convex and contains the origin, the simplex $S$ contains all the points of the segments $[0,x]$. 
\end{proof}

\begin{defn}
	
	Let $L$ be the list of indices of the common vertices of $S_1$ and $S_2$. Then the two simplices $S_1$ and $S_2$ intersect trivially if $S_1\cap S_2$ is contained in ${\rm Span}\left(Z_i,{\rm \, for \,} i \in L\right) \subset V_Y$. 
	
	Similarly, the two cones $\mathcal{C}_{S_1}$ and $\mathcal{C}_{S_2}$ associated to $S_1$ and $S_2$ intersect trivially if $\mathcal{C}_{S_1}\cap \mathcal{C}_{S_2}$ is contained in ${\rm Span}\left(Z_i,{\rm \, for \,} i \in L\right)$.
\end{defn}

\begin{example}
	The two simplices of the main lemma intersect trivially. On the other hand, the two simplices of Figure~\ref{fig:sub1} do not intersect trivially. 
\end{example} 

\begin{lem}
	\label{lem:trivial_intersection_cones}
	Two simplices $S_1$ and $S_2$ of $S_{\mathfrak{C}}$ intersect trivially if and only if $\mathcal{C}_{S_1}$ and $\mathcal{C}_{S_2}$ intersect trivially.
\end{lem}
\begin{proof}
	The forward direction is clear. The backward direction follows from Lemma~\ref{lem:ray_origin}. 
\end{proof}

The main ingredient of the proof of assertion $1$ and $2$ when $\dim\left(\mathfrak{C}\right)<m$ is the following lemma.

\begin{lem}
	\label{lem:trivial_intersection}
	Let $\mathfrak{C}$ be a cell containing the origin of $V_Y$. Then, any two simplices of $S_{\mathfrak{C}}$ intersect trivially.
\end{lem}

\begin{proof}
	Suppose $\mathfrak{C}$ is of dimension $d$. Let $1\leq a\leq m-d$, then we prove by induction on $a$ that: each couple $\left(S,S^{'}\right)$ of simplices of $S_{\mathfrak{C}}$ with $m+1-a$ common vertices intersect trivially. Note that $S$ and $S^{'}$ cannot have less than $d+1$ common vertices since they are ancestors of $\mathfrak{C}$ and thus they contain the $d+1$ vertices of $\mathfrak{C}$.
			
	\subparagraph{Initialization.}
			
	If $a=1$, then $S$ and $S^{'}$ intersect trivially by the main
	lemma.
			
	\subparagraph{Heredity.}
			
	Let $a$ be such that $1\leq a\leq m-d$. Let $S$ and $S^{'}$ be two simplices of $S_{\mathfrak{C}}$ sharing exactly $m+1-a$ vertices, and let $L$ be the list of indices of these vertices. Let $\left(L,i_{1},\dots,i_{a}\right)$ be the list of indices of vertices of $S$. Denote by $\bar{i}_{j}$, for $1\leq j\leq a$, the conjugate vertex of $i_{j}$ relative to	$S$. It follows from Proposition~\ref{prop: parent simplices} that the list of vertices of $S^{'}$ is $\left(L,\bar{i}_{1},\dots,\bar{i}_{a}\right)$.
			
	We now replace the study of intersection of simplices by the intersection of cones. 
	We first associate to any simplex $S\left(L,\alpha_{1},\dots,\alpha_{a}\right)$ of $S_{\mathfrak{C}}$, where $\alpha_{j}\in\left\{ i_{j},\bar{i}_{j}\right\} $, the cone $\mathcal{C} \left(L,\alpha_{1},\dots,\alpha_{a}\right)$ of $\mathbb{R}^m$ generated by the vertices of $S\left(L,\alpha_{1},\dots,\alpha_{a}\right)$. Furthermore, let $p:\mathbb{R}^{m}\rightarrow\mathbb{R}^{a}$ be the projection relative to the vector space ${\rm Span}\left(Z_{i},\,{\rm for\,}i\in L\right)$. We associate to  $S\left(L,\alpha_{1},\dots,\alpha_{a}\right)$ the projection of its corresponding cone, we denote this $a$-dimensional convex cone by
	$$
		\mathscr{C}\left(\alpha_{1},\dots,\alpha_{a}\right):=p\left(\mathcal{C} \left(L,\alpha_{1},\dots,\alpha_{a}\right) \right) \subset \mathbb{R}^{a}. 
	$$
	Clearly, $\mathscr{C}\left(\alpha_{1},\dots,\alpha_{a}\right)$ is the convex cone of $\mathbb{R}^{a} $ generated by $p\left(Z_{\alpha_{1}}\right),\dots,p\left(Z_{\alpha_{a}}\right)$. Since each simplex of $S_{\mathfrak{C}}$ is full-dimensional (Lemma~\ref{lem: full-dimensional}), the cone $\mathscr{C}\left(\alpha_{1},\dots,\alpha_{a}\right)$ is indeed of dimension $a$. The ray generated by $p\left(Z_{\alpha_{j}}\right)$ of $\mathscr{C}\left(\alpha_{1},\dots,\alpha_{a}\right)$ is called an extreme ray of index $\alpha_{j}$. We naturally define the trivial intersection of projected cones by: two projected cones $\mathscr{C}\left( \alpha_{1},\dots,\alpha_{a} \right)$ and $\mathscr{C}\left( \beta_{1},\dots,\beta_{a} \right)$, where $\alpha_j$ and $\beta_j$ belong to $\left\{i_j,\overline{i_j}\right\}$, with $\tilde{L}$ being the list of indices of their common extreme rays, intersect trivially if $\mathscr{C}\left( \alpha_{1},\dots,\alpha_{a} \right)\cap\mathscr{C}\left( \beta_{1},\dots,\beta_{a} \right) \subset {\rm Span}\left(p\left(Z_{i}\right),\,{\rm for\,}i\in\tilde{L}\right)$.
	\begin{claim}
		Two simplices $S_1$ and $S_2$ of $S_{\mathfrak{C}}$ intersect trivially if and only if the projected cones $p\left( \mathcal{C}_{S_1} \right)$ and $p\left( \mathcal{C}_{S_2} \right)$ intersect trivially. 
	\end{claim}
	Indeed, by Lemma~\ref{lem:trivial_intersection_cones} it suffices to show that $\mathcal{C}_{S_1}$ and $\mathcal{C}_{S_2}$ trivially intersect  if and only if $p\left( \mathcal{C}_{S_1} \right)$ and $p\left( \mathcal{C}_{S_2} \right)$  intersect trivially, which is clear from the definitions.
	
	We are going to show that two projected cones associated to two simplices of $S_{\mathfrak{C}}$ intersect trivially. This will prove the lemma. 
	
	\begin{claim}
		\label{claim: Vertex in siplex impossible}The cone $\mathscr{C}\left(i_{1},\dots,i_{a}\right)$ (resp. $\mathscr{C}\left(\bar{i}_{1},\dots,\bar{i}_{a}\right)$ ) cannot contain an  extreme ray of $\mathscr{C}\left(\bar{i}_{1},\dots,\bar{i}_{a}\right)$ (resp. $\mathscr{C}\left(i_{1},\dots,i_{a}\right)$).
	\end{claim}
	Indeed, suppose $p\left(Z_{i_{j}}\right)$ belongs to $\mathscr{C}\left(\bar{i}_{1},\dots,\bar{i}_{a}\right)$,
	for some $1\leq j\leq a$. Then the cone \\$\mathscr{C}\left(\bar{i}_{1},\dots,\hat{\bar{i}}_{j},\dots,\bar{i}_{a},i_{j}\right)$
	can either be of dimension strictly lower than $a$ or it intersects
	nontrivially with $\mathscr{C}\left(\bar{i}_{1},\dots,\bar{i}_{a}\right)$.
	Both cases are forbidden, the first one because the corresponding simplex
	would be flattened and the second one because $S\left(L,\bar{i}_{1},\dots,\bar{i}_{a}\right)$
	and $S\left(L,\bar{i}_{1},\dots,\hat{\bar{i}}_{j},\dots,\bar{i}_{a},i_{j}\right)$ would intersect nontrivially, which is forbidden by the induction hypothesis. This proves Claim~\ref{claim: Vertex in siplex impossible}. 
			
	Let $x\in\mathscr{C}\left(i_{1},\dots,i_{a}\right)\cap\mathscr{C}\left(\bar{i}_{1},\dots,\bar{i}_{a}\right)$ and suppose $x$ is not the origin. Since $\mathscr{C}\left(i_{1},\dots,i_{a}\right)$ is convex, the segment $\left[p\left(Z_{i_{1}}\right),x\right]$ is contained in $\mathscr{C}\left(i_{1},\dots,i_{a}\right)$. It follows from Claim~\ref{claim: Vertex in siplex impossible} that the segment $\left[p\left(Z_{i_{1}}\right),x\right]$ cannot be contained in $\mathscr{C}\left(\bar{i}_{1},\dots,\bar{i}_{a}\right)$, otherwise $p\left(Z_{i_{1}}\right)\in\mathscr{C}\left(\bar{i}_{1},\dots,\bar{i}_{a}\right)$. Thus, there is a point $y$ in $\left[p\left(Z_{i_{1}}\right),x\right]$ belonging to the boundary of $\mathscr{C}\left(\bar{i}_{1},\dots,\bar{i}_{a}\right)$. Then $y$ belongs to a facet (that is a codimension-one face) of $\mathscr{C}\left(\bar{i}_{1},\dots,\bar{i}_{a}\right)$, we denote this facet by $F$. Let $\overline{I}$ be the set of indices of the vertices generating $F$, i.e. $F=\spanp{p\left(Z_{\bar{i}}\right),\bar{i}\in\bar{I}}$. Let $\bar{i}_{j}$, for $1\leq j\leq a$, be the index such that $\left\{ \bar{i}_{1},\dots,\bar{i}_{a}\right\} =\bar{I}\cup\left\{ \bar{i}_{j}\right\} $. Then $\mathscr{C}\left(i_{1},\dots,i_{a}\right)$ and $\mathscr{C}\left(\bar{I},i_{j}\right)$ intersect in $y$. Moreover, this intersection is nontrivial since otherwise $y\in{\rm Span}_{>0}\left(p\left(Z_{i_{j}}\right)\right)$, but then $p\left(Z_{i_{j}}\right)$ belongs to $\mathscr{C}\left(\bar{i}_{1},\dots,\bar{i}_{a}\right)$, which is forbidden by Claim~\ref{claim: Vertex in siplex impossible}. The nontrivial intersection of $\mathscr{C}\left(i_{1},\dots,i_{a}\right)$ and $\mathscr{C}\left(\bar{I},i_{j}\right)$ is forbidden by the induction hypothesis. Hence $\mathscr{C}\left(i_{1},\dots,i_{a}\right)\cap\mathscr{C}\left(\bar{i}_{1},\dots,\bar{i}_{a}\right)=\left\{ 0\right\} $, that is their intersection is trivial. This ends the proof of the induction and thus of the lemma.
\end{proof}

\subparagraph{Proving assertion $2$ when $\dim\left(\mathfrak{C}\right)<m$.}
	
The cell $\mathfrak{C}$ is an internal cell since $\pi_{Y}\left(Y\right)$ belongs to $\mathfrak{C}$. We prove that two cells of $L_{\mathfrak{C}}$ cannot intersect. Suppose that $\mathfrak{C}_{1}$ and $\mathfrak{C}_{2}$ are two cells of $L_{\mathfrak{C}}$ with a common point, then $\mathfrak{C}_{1}$ and $\mathfrak{C}_{2}$ cannot belong to the same simplex, otherwise this simplex is flattened and this is forbidden by Lemma~\ref{lem: full-dimensional}. Hence $\mathfrak{C}_{1}$ and $\mathfrak{C}_{2}$ are descendant cells of two different simplices of $S_{\mathfrak{C}}$, say $S_1$ and $S_2$. Hence, if $\mathfrak{C}_{1}$ and $\mathfrak{C}_{2}$ intersect, then $S_1$ and $S_2$ intersect nontrivially. This is excluded by Lemma~\ref{lem:trivial_intersection}.

\subparagraph{Proving assertion $1$ when $\dim\left(\mathfrak{C}\right)<m$.}
Let $S_i$ be a simplex of $S_{\mathfrak{C}}.$ We recall that $\mathcal{C}_{S_{i}}$ is the convex cone of $V_Y \simeq \mathbb{R}^{m}$ generated by the vertices of $S_{i}$. 
\begin{claim}
The set $U_{\mathfrak{C}}=\bigcup_{S_{i}\in S_{\mathfrak{C}}}S_{i}$ is a neighborhood of $\mathfrak{C}$ if and only if $\bigcup_{S_{i}\in S_{\mathfrak{C}}}\mathcal{C}_{S_{i}}=\mathbb{R}^{m}.$
\end{claim}
If  $\bigcup_{S_{i}\in S_{\mathfrak{C}}}S_{i}$ is a neighborhood of $\mathfrak{C}$, it is also a neighborhood of the origin, then the cone generated by this set is $\bigcup_{S_{i}\in S_{\mathfrak{C}}}\mathcal{C}_{S_{i}}=\mathbb{R}^{m}$. This proves the forward direction. To prove the backward direction, suppose that there exists a point $p$ of $\mathfrak{C}$ in the boundary of $U_{\mathfrak{C}}$. First, this point $p$ cannot be the origin, since we can use Lemma~\ref{lem:ray_origin} to construct a neighborhood of the origin in $U_{\mathfrak{C}}$. Then, if $p \in \mathfrak{C}$ is in the boundary of $U_{\mathfrak{C}}$, we also have that the segment $\left[0,p\right]$ is in the boundary of $U_{\mathfrak{C}}$. Indeed, otherwise the segment $\left]0,p\right]$ is cut by a facet (codimension one face) of a simplex $S$. But $S$ is also an ancestor of $\mathfrak{C}$. Thus the cell $\mathfrak{C}$ intersects another cell of $S$, and thus $S$ is flattened, which is ruled out by Lemma~\ref{lem: full-dimensional}.

We now show that the cone $\mathcal{C}_{U_{\mathfrak{C}}}:=\bigcup_{S_{i}\in S_{\mathfrak{C}}}\mathcal{C}_{S_{i}}$ has no facet (codimension one face). Since $\mathcal{C}_{U_{\mathfrak{C}}}\neq\left\{ 0\right\} $ (it is the union of $m$-dimensional cones) we deduce that $\mathcal{C}_{U_{\mathfrak{C}}}=\bigcup_{S_{i}\in S_{\mathfrak{C}}}\mathcal{C}_{S_{i}}=\mathbb{R}^{m}$. A point in the relative interior of a facet of $\mathcal{C}_{U_{\mathfrak{C}}}$ is a point in the relative interior of a facet of $\mathcal{C}_{S_{i}}$, for some $S_{i}\in S_{\mathfrak{C}}$. Suppose that such a point, say $p$, exists. Since two cones intersect trivially (Lemma~\ref{lem:trivial_intersection_cones} and Lemma~\ref{lem:trivial_intersection}), when a path leaves $C_{S_{i}}$ through $p$, it lands in the conjugate cone relative to this facet. Thus, the point $p$ does not belong to the boundary of $\mathcal{C}_{U_{\mathfrak{C}}}$, and thus it does not exist. This ends the proof.

\subsection{Constancy of the crossing number in $\amplio \times \matp$\label{subsec:Proving the crossing thm}}
	
Since $\grsp{k,n}$ and $\matp$ are path-connected (see \cite{postnikov2006total} for the positive Grassmannian), there exists a continuous
path
\[
	\left(C,\mathcal{Z}\right):\left[0,1\right]\rightarrow\grsp{k,n}\times\matp
\]
between any two couples of points in $\grsp{k,n}\times\matp$.
We also denote $Y\left(t\right)=C\left(t\right)\mathcal{Z}\left(t\right)$.
We prove that the crossing number $c_{n,k,m}\left(Y\left(t\right),\mathcal{Z}\left(t\right)\right)$
is constant along these paths, hence proving the first part of Theorem~\ref{thm: crossing is constant}.
For every $t\in\left[0,1\right]$ , we denote by $\mathscr{L}\left(t\right)$
the list of cells containing the origin $\pi_{Y\left(t\right)}\left(Y\left(t\right)\right)$
of $V_{Y\left(t\right)}$. In the following, we prove that for any
$t_{0}\in\left[0,1\right]$, there exists $\epsilon>0$ such that
$c_{n,k,m}\left(Y\left(t\right),\mathcal{Z}\left(t\right)\right)={\rm Card}\left(\mathscr{L}\left(t\right)\right)$
is constant for $t\in\left]t_{0}-\epsilon,t_{0}+\epsilon\right[\cap\left[0,1\right]$.
Since $\left[0,1\right]$ is compact, we can extract a finite number
of such balls to cover it, hence we deduce that the crossing number
is constant on $\left[0,1\right]$.
	
\paragraph{Step 1.}
	
Fix $t_{0}\in\left[0,1\right]$ and let $\mathscr{L}\left(t_{0}\right)=\left\{ \mathfrak{C}_{1},\dots,\mathfrak{C}_{n_{0}}\right\} $ be the list of cells containing the origin at $t_{0}$. When $\left(Y,\mathcal{Z}\right)$ moves continuously in $\amplio\times\matp$, then the vertices $Z_{i}\left(t\right)=\pi_{Y\left(t\right)}\left(\mathcal{Z}_{i}\right)$ and thus the cells move continuously. By applying Proposition~\ref{prop:
	Config adjacent cells} and Lemma~\ref{lem: full-dimensional} to $\mathfrak{C}_{i}$ at $t_{0}$ we deduce by continuity that there exists $\epsilon_{i}$ such that for every $t\in\left]t_{0}-\epsilon_{i},t_{0}+\epsilon_{i}\right[$:
\begin{itemize}
	\item the set $U_{\mathfrak{C}_{i}}$ is still a neighborhood of $\mathfrak{C}_{i}$,
	\item the cells of $L_{\mathfrak{C}_{i}}$ do not intersect,
	\item the origin $\pi_{Y\left(t\right)}\left(Y\left(t\right)\right)$ belongs to a unique cell of $L_{\mathfrak{C}_{i}}$,
	 \item each ancestor simplex of $\mathfrak{C}_i$ is full-dimensional. 
\end{itemize}
We denote by $\mathfrak{c}_{i}\left(t\right)$ the cell of $L_{\mathfrak{C}_{i}}$ containing
the origin at time $t$. It is such that $\mathfrak{c}_{i}\left(t_0\right) = \mathfrak{C}_{i}$.  We define such $\epsilon_{i}$ for every
$1\leq i\leq n_{0}$ and let $\epsilon:=\min\left\{ \epsilon_{1},\dots,\epsilon_{n_{0}}\right\} $.
Hence we defined $n_{0}$ paths
\begin{align*}
	\left]t_{0}-\epsilon,t_{0}+\epsilon\right[ & \rightarrow\mathfrak{Cells}                \\
	t                                          & \rightarrow \mathfrak{c}_{i}\left(t\right) 
\end{align*}
which are pairwise disjoint at $t_{0}$. Each path has a finite number of discontinuous points corresponding to the origin jumping from one cell to another. We redefine $\epsilon>0$ small enough such that each path is continuous on $]t_0-\epsilon,t_0+\epsilon[$ except possibly at $t_0$. 
	
\paragraph{Step 2.}
	
We prove that, up to choosing $\epsilon>0$ small enough, any two paths $t\rightarrow \mathfrak{c}_{i}\left(t\right)$ constructed in Step~$1$ do not intersect on $\left]t_{0}-\epsilon,t_{0}+\epsilon\right[$. We first analyze the situation when $t \in [t_0,t_0+\epsilon[$. 

Suppose that the path $t \rightarrow \mathfrak{c}_i$, for  $t \in [t_0,t_0+\epsilon[$, is discontinuous at $t_0$. Denote by $\mathfrak{C}_{i}^{+}$ the value of $t\rightarrow \mathfrak{c}_{i}\left(t\right)$ on $\left]t_{0},t_{0}+\epsilon\right[$. 
\begin{claim}
	The cell $\mathfrak{C}^{+}_{i}$ is an ancestor cell of $\mathfrak{C}_{i}$. 
\end{claim}

Indeed, we first recall (Step~$1$) that each ancestor simplex of $\mathfrak{C}_{i}$ is full-dimensional when $t\in\left]t_{0}-\epsilon,t_{0}+\epsilon\right[$. According to  the type of discontinuity of $t\rightarrow\mathfrak{c}_{i}\left(t\right)$ at $t_{0}$, we deduce that the origin $\pi_{Y(t_0)}(Y(t_0))$ must belong to the closure of $\mathfrak{C}^{+}_{i}$ at $t_0$. Moreover, the closure of a full-dimensional cell equals the disjoint union of this cell and its descendent cells. Thus, if $\pi_{Y(t_0)}(Y(t_0))$ belongs to the closure of $\mathfrak{C}^{+}_{i}$ at $t_0$, then $\mathfrak{C}^{+}_{i}$ is in the lineage of $\mathfrak{C}_{i}$, and even more $\mathfrak{C}^{+}_{i}$ must be an ancestor of $\mathfrak{C}_{i}$. 
This ends the proof of the claim. 

Fix $i\neq j$ and suppose that the paths $t\rightarrow \mathfrak{c}_{i}\left(t\right)$ and $t\rightarrow \mathfrak{c}_{j}\left(t\right)$ intersect on $[t_0,t_0+\epsilon[$:
\begin{itemize}
	\item if $t\rightarrow \mathfrak{c}_{i}\left(t\right)$ and $t\rightarrow \mathfrak{c}_{j}\left(t\right)$ are discontinuous at $t_0$, then $\mathfrak{C}_{i}^{+}= \mathfrak{C}_{j}^{+}=: \mathfrak{C} $, and then $\mathfrak{C}$ is
	      an ancestor cell of $\mathfrak{C}_{i}$ and of $\mathfrak{C}_{j}$, so $\mathfrak{C}_{i}$ and $\mathfrak{C}_{j}$
	      belong to the same simplex. But then, at $t_{0}$ the origin belongs
	      to two cells of the same simplex, hence this simplex is flattened, which is impossible by Lemma~\ref{lem: full-dimensional}.
	\item if $t\rightarrow \mathfrak{c}_{i}\left(t\right)$ is discontinuous and $t\rightarrow \mathfrak{c}_{j}\left(t\right)$ is continuous at $t_0$, then $\mathfrak{C}_{i}^{+}= \mathfrak{C}_{j}$ and then $\mathfrak{C}_{i}$ and $\mathfrak{C}_{j}$ are in the same lineage. But then at $t_{0}$ the origin is
	      contained in two cells in the same lineage, which is impossible since by definition the intersection of two cells in the same lineage is empty. The situation is similar if we exchange $i$ and
	      $j$.
\end{itemize}
In the last case, when  $t\rightarrow \mathfrak{c}_{i}\left(t\right)$ and $t\rightarrow \mathfrak{c}_{j}\left(t\right)$ are continuous at $t_0$, it suffices to shorten the interval $\left[t_{0},t_{0}+\epsilon\right[$ by choosing $\epsilon>0$ small enough to avoid the intersection.  By doing a similar reasoning for $t\in\left]t_{0}-\epsilon,t_{0}\right]$, we prove that there exists $\epsilon>0$ such that the paths $t\rightarrow \mathfrak{C}_{i}\left(t\right)$, for $1\leq i\leq n_{0}$ do no intersect on $\left]t_{0}-\epsilon,t_{0}+\epsilon\right[$.
			
\paragraph{Step 3.}
			
We now prove that for $t\in\left]t_{0}-\epsilon,t_{0} + \epsilon\right[$, we have
\[
	\mathcal{L}\left(t\right)=\left\{ \mathfrak{c}_{1}\left(t\right),\dots,\mathfrak{c}_{n_{0}}\left(t\right)\right\} ,
\] up to choosing $\epsilon>0$ small enough. Since we showed that the paths $t \rightarrow \mathfrak{c}_i(t)$ do no intersect, this concludes the proof. By construction, we have $\mathcal{L}\left(t_0 \right)=\left\{ \mathfrak{c}_{1}\left(t_0 \right),\dots,\mathfrak{c}_{n_{0}}\left(t_0 \right)\right\}$ and  the inclusion $\left\{ \mathfrak{c}_{1}\left(t\right),\dots,\mathfrak{c}_{n_{0}}\left(t\right)\right\} \subseteq\mathcal{L}\left(t\right)$ for $t \in ]t_0-\epsilon,t_0+\epsilon[$. Suppose that the inclusion is strict, then up to choosing $\epsilon>0$ small enough, this amounts to assuming that there exist $\mu>0$ and a cell $\tilde{\mathfrak{C}}$ not in the set $\left\{ \mathfrak{c}_{1}\left(t\right),\dots,\mathfrak{c}_{n_{0}}\left(t\right)\right\}$, for $t \in ]t_0,t_0+\epsilon[$ such that the origin is contained in $\tilde{\mathfrak{C}}$ when $t \in ]t_0,t_0+\mu[$. There are two possibilities.
\begin{itemize}
	\item If $\tilde{\mathfrak{C}}$ is a cell of a simplex $\tilde{S}$ which is not an ancestor of a cell of $\mathcal{L}\left(t\right)$ for $t \in ]t_0,t_0+\epsilon[$. Then, the cell $\tilde{\mathfrak{C}}$ belongs to the boundary of $\tilde{S}$, which is a closed set of $V_{Y(t)}$.
		But this is impossible since the origin is contained in $\tilde{\mathfrak{C}}$ when $t \in ]t_0,t_0+\mu[$, and the interval  $]t_0,t_0+\mu[$ is open on the left side.
	\item If $\tilde{\mathfrak{C}}$ is a cell of an ancestor simplex of a cell, say $\mathfrak{C_i}$, of $\mathcal{L}\left(t\right)$  for $t \in ]t_0,t_0+\epsilon[$. Then, for $t$ close enough to $t_0$ the origin belongs to two cells of the same simplex: $\tilde{\mathfrak{C}}$ and $\mathfrak{C}_i$. But then this simplex is flattened and this is ruled out by Lemma~\ref{lem: full-dimensional}. 
\end{itemize}
	
By doing a similar reasoning for $t<t_{0}$, we conclude that there exists $\epsilon$ such that $\mathcal{L}\left(t\right)=\left\{ \mathfrak{c}_{1}\left(t\right),\dots,\mathfrak{c}_{n_{0}}\left(t\right)\right\} $ for $t\in\left]t_{0}-\epsilon,t_{0}+\epsilon\right[$. This ends the proof.
      							      	
      	\subsection{Independence of the crossing number in $n$\label{subsec:Independence n crossing}}
      							      	
Let $m=2r-1$ for some positive integer $r$. We use the same approach as for the winding number. We showed that the crossing number is independent of the point in $\amplio \times \matp$. We now show that for two specific choices of couples, one in $\mathcal{A}_{n+1,k,m}^{>0} \times {\rm Mat}_{n+1,k+m}^{>0}$ and the other one in $\mathcal{A}_{n,k,m}^{>0} \times \matp$ the crossing number is the same. Thus the crossing number is independent of $n$. 
\begin{prop}
There exist $\left(Y,\mathcal{Z}^{'}\right)\in\mathcal{A}_{n+1,k,m}^{>0} \times {\rm Mat}_{n+1,k+m}^{>0}$ and $\left(Y,\mathcal{Z}\right)\in\mathcal{A}_{n,k,m}^{>0}\times \matp$ such that 
\[
c_{n+1,k,m}\left(Y,\mathcal{Z}^{'}\right)=c_{n,k,m}\left(Y,\mathcal{Z}\right).
\]
\end{prop}
\begin{proof}
Let $\mathcal{Z}^{'}=\left(\mathcal{Z}_{1},\dots,\mathcal{Z}_{n+1}\right)\in{\rm Mat}_{n+1,k+m}^{>0}$ and then define $\mathcal{Z}=\left(\mathcal{Z}_{1},\dots,\mathcal{Z}_{n}\right)\in\matp$. We first show that $\mathcal{A}_{n,k,m}^{>0}\cap\mathcal{A}_{n+1,k,m}^{>0}\neq \varnothing$. Indeed, let $C\in\grsp{k,n}$ so that $Y=C\mathcal{Z}\in\amplio$. We define the matrix $C^{'}$ by adding a $\left(n+1\right)$th column of zeros to $C$, hence we have $Y=C\mathcal{Z}=C^{'}\mathcal{Z}^{'}$, and thus $Y\in\mathcal{A}_{n,k,m}^{>0}\cap\mathcal{A}_{n+1,k,m}$. Now,  $\amplio$ and $\mathcal{A}_{n+1,k,m}^{>0}$ are open sets of ${\rm Gr}_{k,k+m}$ (see the proof of Lemma~9.4 in \cite{galashin2020parity} for an argument), moreover since the closure of ${\rm Gr}_{k,n+1}^{>0}$ equals ${\rm Gr}_{k,n+1}^{\geq 0}$ it follows that the closure of $\mathcal{A}_{n+1,k,m}^{>0}$ is the amplituhedron $\mathcal{A}_{n+1,k,m}$. Hence we conclude that $\mathcal{A}_{n,k,m}^{>0}\cap\mathcal{A}_{n+1,k,m}^{>0}\neq \varnothing$. 

Then, choose $Y_{2}$ in $\mathcal{A}_{n,k,m}^{>0}\cap\mathcal{A}_{n+1,k,m}^{>0}$ such that 
\[
\left\langle Y_{2},i_{1},i_{1}+1,\dots,\widehat{i_{j}+\epsilon},\dots,i_{r},i_{r}+1\right\rangle \neq0,
\]
where $\left(i_{1},i_{1}+1,\dots,i_{r},i_{r}+1\right)\in{\left[n+1\right] \choose m+1}$, $1\leq j \leq r$ and $\epsilon \in {0,1}$. This is always possible since the dimension of $\mathcal{A}_{n,k,m}^{>0}\cap\mathcal{A}_{n+1,k,m}^{>0}$ is $km$ since it is an intersection of open sets, and the vanishing locus of the twistor coordinates $\left\langle Y_{2},i_{1},i_{1}+1,\dots,\widehat{i_{j}+\epsilon},\dots,i_{r},i_{r}+1\right\rangle =0$ is of codimension $1$. Thus, if the origin of $V_{Y_{2}}$ is contained in a simplex, it is contained in its $m$-dimensional cell.

In order to prove that the crossing number $c_{n,k,m}\left(Y_{2},\mathcal{Z}\right)$ and $c_{n+1,k,m}\left(Y_{2},\mathcal{Z}^{'}\right)$ is the same it suffices to prove that the origin cannot belong to a cell
\[
C\left(i_{1},i_{1}+1,\dots,i_{r-1}+1,n,n+1\right),
\]
where $\left(i_{1},i_{1}+1,\dots,i_{r-1}+1\right)\in{\left[n-1\right] \choose m-1}$ . However, if such a cell contains the origin then 
\[
{\rm sign}\left\langle Y_{2},i_{1},i_{1}+1,\dots,i_{r-1}+1,\widehat{n},n+1\right\rangle =-{\rm sign}\left\langle Y_{2},i_{1},i_{1}+1,\dots,i_{r-1}+1,n,\widehat{n+1}\right\rangle 
\]
and it follows from the strict coarse boundary conditions for $Y_{2}\in\mathcal{A}_{n+1,k,m}^{>0}$ and for $Y_{2}\in\amplio$ that the twistor coordinates on LHS and on RHS are positive, and thus the equality cannot hold. 
\end{proof}
      							      	
      	\subsection{The crossing number for $n=k+m$ \label{subsec: crossing for n=00003Dk+m}}
      							      	
      	Let $m=2r-1$ for some positive integer $r$. We show that there exists
      	$\mathcal{Z}\in{\rm Mat}_{n,n}^{>0}$ and $C\in\grsp{k,k+m}$ such
      	that the crossing number
      	\begin{equation}
      		c_{n=k+m,k,m}\left(C\mathcal{Z},\mathcal{Z}\right)=\begin{cases}
      		\frac{2k+m-1}{m+1}{\frac{k+m-2}{2} \choose \frac{m-1}{2}} & {\rm for\;}k\;{\rm odd},\\
      		2{\frac{k+m-1}{2} \choose \frac{m+1}{2}} & {\rm for\;}k\;{\rm even.}
      		\end{cases}\label{eq: explicit formula crossing}
      	\end{equation}
      	Since we proved that the crossing number is independent of $\mathcal{Z},Y$
      	and $n$, this ends the proof of Theorem~\ref{thm: crossing is constant}.
      							      	
      	\paragraph{Step $1.$}
      							      	
      	Each simplex containing the origin of $V_{Y}$ is full-dimensional by Lemma~\ref{lem: full-dimensional}, so we can choose $C\in\grsp{k,n}$ such that the origin of $V_{Y}$ only intersects these simplices in their interior. Equivalently, we pick $C\in\grsp{k,n}$ such that each twistor coordinate
      	\\$\left\langle C\mathcal{Z},i_{1},i_{1}+1,\dots,\widehat{i_{j}+\epsilon},\dots,i_{r},i_{r}+1\right\rangle $,
      	for $\epsilon\in\left\{ 0,1\right\} $ and $\left(i_{1},i_{1}+1,\dots,i_r,i_{r}+1\right)\in{\left[n\right] \choose m+1}$, which is always possible since the locus of points with one vanishing
      	twistor coordinate is of codimension $1$ in $\grsp{k,n}$.
      							      	
      	\paragraph{Step $2$.}
      							      	
      	Let $\left(i_{1},i_{1}+1,\dots,i_{r},i_r,i_{r}+1\right)\in{\left[n\right] \choose m+1}$, then the origin belongs to the simplex $S(i_{1},i_{1}+1,\ldots,i_{r},i_{r}+1)$
      	if and only if it satisfies the following two conditions:
      	\begin{equation}
      		\tag{Condition \ensuremath{\left(i\right)}}{\rm sign}\left\langle i_{1},i_{1}+1,\dots,\widehat{i_{j}+1},\dots,i_{r},i_{r}+1\right\rangle =-{\rm sign}\left\langle i_{1},i_{1}+1,\dots,\widehat{i_{j}},\dots,i_{r},i_{r}+1\right\rangle ,\label{eq: condition i}
      	\end{equation}
      	for $1\leq j\leq r$, and
      							      	
      	\begin{equation}
      		\tag{Condition \ensuremath{\left(ii\right)}}{\rm sign}\left\langle i_{1},i_{1}+1,\dots,\widehat{i_{j}},\dots,i_{r},i_{r}+1\right\rangle \,{\rm is\,independent\,of\,}j.\label{eq:condition ii}
      	\end{equation}
      	Since $n=k+m$ the twistor coordinates are given by determinants of
      	square matrices, so we have
      	\[
      		\left\langle i_{1},i_{1}+1,\dots,\widehat{i_{j}+\epsilon},\dots,i_{r},i_{r}+1\right\rangle =\det\left(\begin{matrix}C\\
      		I_{i_{1},i_{1}+1,\ldots,\widehat{i_{j}+\epsilon},\ldots,i_{r},i_{r}+1}
      		\end{matrix}\right)\det(\mathcal{Z}),
      	\]
      	where $I_{i_{1},i_{1}+1,\ldots,\widehat{i_{j}+\epsilon},\ldots,i_{r},i_{r}+1}$
      	is the $m\times\left(k+m\right)$ matrix whose $l$th row has a $1$
      	at the $l$th index of the list $i_{1},i_{1}+1,\ldots,\widehat{i_{j}+\epsilon},\ldots,i_{r},i_{r}+1$
      	and zeros elsewhere. Using the standard expansion of the determinant
      	we get
      	\[
      		\det\left(\begin{matrix}C\\
      		I_{i_{1},i_{1}+1,\ldots,\widehat{i_{j}+\epsilon},\ldots,i_{r},i_{r}+1}
      		\end{matrix}\right)=\left(-1\right)^{k+i_{j}-\epsilon}\det\left(C_{[n]\setminus\{i_{1},i_{1}+1,\ldots,\widehat{i_{j}+\epsilon},\ldots,i_{r},i_{r}+1\}}\right).
      	\]
      	Hence, since $\mathcal{Z}\in{\rm Mat}_{n,n}^{>0}$ and $C\in\grsp{k,k+m}$
      	we obtain 
      	\[
      		{\rm sign}\left\langle i_{1},i_{1}+1,\dots,\widehat{i_{j}+\epsilon},\dots,i_{r},i_{r}+1\right\rangle =\left(-1\right)^{k+i_{j}+\epsilon}.
      	\]
      	Thus, condition $\left(i\right)$ is always satisfied and condition
      	$\left(ii\right)$ is satisfied precisely if
      	\begin{equation}
      		i_{1}=i_{2}=\cdots=i_{r}\mod2.\label{eq:condition of giving crossing}
      	\end{equation}
      	Thus the crossing number is equal to the number of sequences $\left(i_{1},\ldots,i_{r}\right)$
      	between $1$ and $n-1$ such that $i_{j+1}>i_{j}+1$ satisfying Eq.~(\ref{eq:condition of giving crossing}).
      	This is exactly the expression of the crossing number given in Eq.~(\ref{eq: explicit formula crossing}).

      	\newpage{}
      							      	
      	\bibliographystyle{alpha}
      	\bibliography{Bibliographie_Amplituhedron}
      							      							      	
\end{document}